\documentclass[12pt,twoside,english,reqno,a4paper]{amsart}

\usepackage{fullpage}

\usepackage{natbib}
\newtheorem{lemma}{Lemma}

\newtheorem{rem}{Remark}[section]
\newcommand{\RomanNumeralCaps}[1]
\linenumbers

\numberwithin{equation}{section}
\usepackage{boondox-cal}

\usepackage[utf8]{inputenc}
\usepackage[english]{babel}
\usepackage[lined,boxed]{algorithm2e}
\usepackage{float} 
\usepackage{hyperref}
\usepackage{listings,graphicx,caption,mwe,amsmath,varioref,amscd,amssymb,color,xcolor,bm,amsthm,amsfonts,graphics,lineno,float,comment}
\usepackage{mathrsfs} 
\usepackage[foot]{amsaddr}

\newcommand{\eps}{\varepsilon}

\newcommand{\R}{{\mathbb{R}}}

\DeclareMathOperator*{\supp}{supp}

\renewcommand{\d}{{\mathrm d}}

\newcommand{\m}{\mathcal m}
\newcommand{\mmu}{\mu_{\textsc{cl}}}

\newcommand{\ignore}[1]{}

\usepackage{color,xcolor}

\newcounter{boxlblcounter}

\newenvironment{boxlabel}
  {\begin{list}
    {\arabic{boxlblcounter}}
    {\usecounter{boxlblcounter}
     \setlength{\labelwidth}{20pt}
     \setlength{\labelsep}{5pt}
     \setlength{\itemsep}{2pt}
     \setlength{\leftmargin}{25pt}
     \setlength{\rightmargin}{0cm}
     \setlength{\itemindent}{0em}
    }
  }
{\end{list}}

\newcounter{boxlblcountertwo}
\newenvironment{boxlabeltwo}
  {\begin{list}
    {\arabic{boxlblcountertwo}}
    {\usecounter{boxlblcountertwo}
     \setlength{\labelwidth}{0pt}
     \setlength{\labelsep}{5pt}
     \setlength{\itemsep}{2pt}
     \setlength{\leftmargin}{5pt}
     \setlength{\rightmargin}{0cm}
     \setlength{\itemindent}{0em}
    }
  }
{\end{list}}

\newcommand{\nocontentsline}[3]{}
\newcommand{\tocless}[2]{\bgroup\let\addcontentsline=\nocontentsline#1{#2}\egroup}

\title{Thin-film equations with singular potentials: an alternative solution to the contact-line paradox}

\author[R. Durastanti]{Riccardo Durastanti$^1$}

\author[L. Giacomelli]{Lorenzo Giacomelli$^{2,*}$}

\address{$^1$
Department of Mathematics and Applications ``Renato Caccioppoli'', University of Naples ``Federico II'', Via Cintia, Monte S. Angelo, 80126 Napoli, Italy
\\ riccardo.durastanti@unina.it}

\address{$^2$ SBAI Department, Sapienza University of Rome, Via Antonio Scarpa 16, 00161 Roma, Italy
\\ lorenzo.giacomelli@uniroma1.it}

\address{$^*$ Corresponding author}

\begin{document}

\begin{abstract}
In the regime of lubrication approximation, we look at spreading phenomena under the action of singular potentials of the form $P(h)\approx h^{1-m}$ as $h\to 0^+$ with $m>1$, modeling repulsion between the liquid-gas interface and the substrate. We assume zero slippage at the contact line. Based on formal analysis arguments, we report that for any $m>1$ and any value of the speed (both positive and negative) there exists a three-parameter, hence generic, family of fronts (i.e., traveling-wave solutions with a contact line). A two-parameter family of advancing ``linear-log'' fronts also exists, having a logarithmically corrected linear behaviour in the liquid bulk. All these fronts have finite rate of dissipation, indicating that singular potentials stand as an alternative solution to the contact-line paradox. In agreement with steady states, fronts have microscopic contact angle equal to $\pi/2$ for all $m>1$ and finite energy for all $m<3$. We also propose a selection criterion for the fronts, based on thermodynamically consistent contact-line conditions modeling friction {\em at} the contact line. So as contact-angle conditions do in the case of slippage models, this criterion selects a {\em unique} (up to translation) linear-log front for each positive speed. Numerical evidence suggests that, fixed the speed and the frictional coefficient, its shape depends on the spreading coefficient, with steeper fronts in partial wetting and a more prominent precursor region in dry complete wetting.
\end{abstract}



\maketitle

\tableofcontents

\renewcommand{\thefootnote}{\arabic{footnote}}

\section{Introduction}

It is about half a century since the no-slip paradox was discovered by \citet{HS} and \citet{DD}. The paradox may be summarized as follows: take a liquid droplet which is sliding over a solid dry substrate, as modeled by Stokes equations;  if the no-slip condition were adopted, that is, if a null liquid's horizontal velocity were prescribed at the liquid-solid interface, then an infinite force would be required to move the {\it contact line}, i.e. the triple junction where solid, liquid and gas meet. In the words of \citet{HS}, ``not even Herakles could sink a solid''.
The paradox may be discussed already in the framework of lubrication approximation: this is an asymptotic limit of the full Navier-Stokes system under a suitable scaling, which in words requires a small ratio of vertical vs horizontal length-scale, a relatively small velocity, and a relatively large surface tension. It is a simplified model which however retains the essential physics at such scales: a dissipative evolution driven by surface tension and limited by both viscous and interfacial friction (see for example \eqref{fgf}, \eqref{GF-stan}, and \eqref{GF-new} below). In fact, generic contact lines, at leading order around one of their points, are locally straight: therefore the essential features of the contact-line paradox are already captured by a one-dimensional setting, which we therefore adopt in the sequel.

\subsection{The model}

In lubrication theory, the evolution of a thin liquid film, or a droplet, of viscous incompressible liquid over a horizontal solid substrate is described by the thin-film equation
\citep{Greenspan,Hoc1,ODB,GO2,KM2}, which in its basic one-dimensional form reads as
\begin{equation}\label{TFE}
h_{t} +(hV)_x=0, \quad V= \frac{\gamma}{\mu}\frac{\m(h)}{h} (h_{xx}-Q'(h))_x \quad\mbox{on $ \{h>0\}$}.
\end{equation}
Here the solid substrate corresponds to the $x$-axis, $t$ is time, $h(t,x)$ is the liquid's height over the solid, $\mu$ is the liquid's viscosity, and $\gamma$ is the liquid-gas surface tension. The mobility function $\m$ depends on the condition at the liquid-solid interface: when the no-slip condition is assumed, then
\begin{equation}\label{h:m}
\m(h)=\tfrac13 h^3.
\end{equation}
The potential $Q$ usually combines the effects of intermolecular, surface, and gravitational forces \citep{DG}; here we shall ignore the latter ones for simplicity:
\begin{equation}\label{def-Q}
Q(h)=(P(h)-S+G(h))\chi_{\{h>0\}}, \qquad \mbox{with $G\equiv 0$ and $S\in \R$ in this manuscript.}
\end{equation}

The constant $S$ (assumed to be relatively small in lubrication theory, cf. Remark \ref{rem:pi}) is the non-dimensional {\em spreading coefficient}:
\begin{equation*}
S=\frac{\mbox{spreading coefficient}}{\gamma}= \frac{\gamma_{SG}-\gamma_{SL}-\gamma}{\gamma} = \frac{\gamma_{SG}-\gamma_{SL}}{\gamma}-1,
\end{equation*}
where $\gamma_{SL}$, and $\gamma_{SG}$ are the solid-liquid and solid-gas tensions, respectively. There is, however, a caveat to be made at this point. In thermodynamic equilibrium of the solid with the surrounding vapor phase (the so-called ``moist'' case, which concerns for instance a surface which has been pre-exposed to vapor), $\gamma_{SG}$ is usually denoted by $\gamma_{SV}$, and its value can never exceed $\gamma_{SL}+\gamma$. Indeed, otherwise the free energy of a solid/vapor interface could be lowered by inserting a liquid film in between: the equilibrium solid/vapor interface would then comprise such film, leading to $\gamma_{SV}=\gamma_{SL}+\gamma$. Therefore, $S\le 0$ in the ``moist'' case. On the other hand, when the solid and the gaseous phases are not in thermodynamical equilibrium (the so-called ``dry'' case), there is no constraint on the sign of $S$.

\smallskip

The function $P$ is an intermolecular potential. Generally speaking, $P$ is singular as $h\to 0^+$ and decays to zero as $h\to +\infty$, with $\Pi=-P'$ usually referred to as the {\it disjoining pressure}. We consider the case in which $P$ is {\it short-range repulsive}, in the sense that it penalizes short distances between the liquid-gas interface and the solid:$^1$\footnotetext[1]{As $h\to h_0$, we write: $f(h)\sim g(h)$ when $f(h)/g(h)\to 1$; $f(h)\approx g(h)$ when $C>0$ exists such that $f(h)/g(h)\to C$; $f(h)=O(g(h))$ when $f(h)/g(h)$ remains bounded; $f(h)=o(g(h))$ when $f(h)/g(h)\to 0$. Also, we write $a\ll b$ if a universal constant $C\ge 1$ exists such that $a\le Cb$.\hfill $\ $}
\begin{equation}\label{hp-p0-}
P(h) \sim  \tfrac{A}{m-1} h^{1-m} \ \mbox{ as $h\to 0^+$}, \quad A>0,\quad m>1, \quad P(0)=P(+\infty)=0.
\end{equation}
The standard choice for $P$ yields $m=3$:
\begin{equation}\label{vdw}
P_0(h) = A_0 h^{-2}, \ \mbox{ $A_0=\frac{A'}{12\pi\gamma}$, \ where $A'>0$ is the Hamaker constant},
\end{equation}
which corresponds to an integration of Lifshitz–van der Waals interactions between molecules \citep{I2011,Craster2009} and which we shall hereafter refer to as van der Waals potentials. There are, however, reasons to examine different values of $m$. The first one is that the form of the disjoining pressure is highly dependent on the nature of the dominant intermolecular force (molecular, electrostatic, structural) and on the scales under consideration: for instance, the electrostatic and structural contributions for water on glass or silica surfaces yield $m=1$ or $m=2$, depending on thickness \citep{Pashley1980,TDS1988}; we refer to the lucid discussion in \citet{DF2018}. A second reason will be introduced in \S\ref{ss:statics}.

\subsection{The contact-line paradox}

In order to introduce the contact-line paradox, it is convenient to describe the basic energetic structure of \eqref{TFE}. The free energy of the system is given by
\begin{equation}\label{def-E}
E[h]= \gamma \int_{\{h>0\}} \left(\tfrac12 h_x^2 +Q(h)\right)\d x = \gamma\int_{\{h>0\}} \left(1+\tfrac12 h_x^2 +(Q(h)-1)\right)\d x.
\end{equation}
In lubrication theory, the term $\gamma(1+\tfrac12 h_x^2)$ is the leading-order approximation of the liquid-gas surface energy density $\gamma\sqrt{1+h_x^2}$. The summand $1$ is incorporated in the potential $Q$. Smooth, positive and, say, periodic solutions to \eqref{TFE} (e.g. modelling a liquid film) satisfy the energy balance
\begin{equation}\label{fgf}
\frac{\d}{\d t} E[h]
= -\underbrace{\mu\int_{\{h>0\}} \frac{h^2}{\m(h)} V^2\d x}_{\textrm{rate of bulk dissipation}}.
\end{equation}
When non-negative solutions are considered (e.g. modelling a droplet), then contributions at the contact line appear (cf. e.g. \eqref{gfpax} below), but the rate of bulk dissipation remains the same: it encodes both viscous friction within the liquid and, through $\m$, interfacial friction at the liquid-solid interface.

\smallskip

In the framework of \eqref{TFE}, the contact-line paradox manifests itself as follows. Assume the no-slip condition, i.e. \eqref{h:m}, and consider a  travelling wave solutions to \eqref{TFE}:
\begin{equation}\label{TW-intro}
h(t,x)=H(y), \quad y=x+Vt, \quad \mbox{with $H>0$ in $(0,+\infty)$ and $H(0)=0$,}
\end{equation}
where $V\ne 0$ is a constant velocity (here we have assumed w.l.o.g. that the contact line is initially located at $x=0$). If $P\equiv 0$, then advancing ($V>0$) travelling wave solutions to \eqref{TFE} of the form \eqref{TW-intro} do not exist, whereas receding ones ($V<0$) have a non-integrable rate of dissipation density (see \eqref{tw-noslip} below).
Therefore, for $\m(h)=\tfrac13 h^3$ and $P\equiv 0$, travelling waves with finite dissipation do not exist at all: this is the manifestation of the contact-line paradox in lubrication theory.

\smallskip

Since the contact-line paradox was discovered, quite a few enrichments of the basic model have being put forward to relieve it: we refer to the reviews by \citet{ODB}, \citet{DG}, \citet{Bonn}, and \citet{A}. The standard method, which was first investigated by \citet{HuhMason1977}, \citet{hocking_1976,hocking_1977}, and \citet{Greenspan}, is to allow for fluid slip over the solid (at least in a neighborhood of the contact line), which amounts to prescribing a relation between horizontal velocity and shear stress at the liquid-solid interface. In lubrication theory, these relations modify the mobility function to $\m(h)=\frac13 (h^3 + \lambda^{3-n} h^n)$ ($n=2$ for the classical \citet{navier} slip condition); $\lambda$ is a length-scale whose inverse is proportional to liquid-solid friction. A second one, introduced to our knowledge by \citet{WS}, is to assume a shear-thinning rheology, with a vanishing liquid's viscosity as the shear stress blows up: this introduces a nonlinear dependence of $V$ on $h_{xxx}$ (see e.g. \citet{FliK,King,AG,AG2}). More recently, it has been observed by \citet{red-col, red-col-2} and \citet{JAP,JanDG} that the Kelvin effect, i.e., a curvature-induced variation of saturation conditions, may also be employed to resolve contact-line paradox: in this case, \eqref{TFE} is complemented with a singular term in non-divergence form.

\smallskip

Another way to resolve the paradox was first discussed by  \citet{Starov}, \citet{DG-CR}, and \citet{HDG}: it consists in taking the effect of intermolecular potentials $P$ into account. The goal of this note is to revisit this phenomenon in a systematic way for generic potentials. To this aim, it is convenient to review the statics first.

\subsection{Statics}\label{ss:statics}

Consider absolute minimizers $h_{min}$ of $E$, as given by \eqref{def-E}, under the constraint of given mass $M$.

\smallskip

When $P\equiv 0$ and $S<0$, $h_{min}$ is an arc of parabola characterized by its mass $M$ and $|S|$; in particular, its slope at the contact line $\partial \{h>0\}$ is determined by $|S|$:
$$
h_{min}=\frac{3M}{4s^3}(s^2-x^2)_+, \quad s^2=\frac{3M}{2\tan\theta_S}, \quad \theta_S:=\arctan\sqrt{2|S|}, \quad S<0.
$$
When $P\equiv 0$ and $S\ge 0$, $h_{min}$ instead does not exist, and minimizing sequences converge to an unbounded film with zero thickness. Therefore it is common to define  the {\it static} (or {\it equilibrium}) {\it microscopic contact angle} $\theta_S$, and to name the two regimes, as follows:
\begin{equation}\label{def-thetaS}
\theta_S:= \left\{\begin{array}{lll} \arctan\sqrt{2|S|} & \mbox{ if $S<0$}  & \qquad \mbox{({\it partial wetting})}
\\[1ex] 0 & \mbox{ if $S\ge 0$} & \qquad \mbox{({\it complete wetting})}. \end{array}\right.
\end{equation}

Let us now take $P$ into account. Consider steady states $h_{\textsc{ss}}$ with connected positivity set, that is, solutions to the Euler-Lagrange equation
\begin{equation}\label{EL}
-h_{xx} + P'(h) = \Lambda,
\end{equation}
where $\Lambda\in \R$ is a Lagrange multiplier coming from the mass constraint. The properties of $h_{\textsc{ss}}$ have been formally discussed by \citet{JDG}, \citet{DG}, and \citet{LJ}.
Their macroscopic shape (here encoded by looking at the regime $M\gg 1$) may be droplet-like or pancake-like. When $R(h)=Q(h)/h$ has a unique absolute minimum point $e_*$, then $e_*$ is characterized by
\begin{equation}\label{e*-DG}
S =P(e_*)-e_* P'(e_*)
\end{equation}
and $h_{\textsc{ss}}$ is pancake-shaped:
\begin{equation}\label{as-p}
h_{\textsc{ss}}\sim e_*\chi_{\{|x|\le s\}},\ \ s \sim\frac{M}{2e_*} \quad \mbox{for $M\gg 1$ } \quad \mbox{if $e_*<+\infty$}.
\end{equation}
Since $S>0$ implies $e_*<+\infty$, this configuration is generic in dry complete wetting. Thus, it is the wetting coefficient which drives the system towards a ``pancake'' equilibrium. If on the other hand $R$ has no absolute minimum, then necessarily $S\le 0$; steady states are droplet-shaped if $S<0$,
\begin{equation}\label{as-d}
h_{\textsc{ss}}\sim\frac{3M}{4s^3}(s^2-x^2)_+, \ \ s^2 \sim \frac{3M}{2\tan\theta_{mac}}, \ \ \theta_{mac}:=\arctan\sqrt{2|S|}\quad \mbox{for $M\gg 1$ } \quad \mbox{if $S<0$},
\end{equation}
where $\theta_{mac}$ denotes the {\em macroscopic} contact angle  (see Fig. \ref{fig:droplet}). When $S=0$ and $e_*=+\infty$ the shape is still droplet-like, but constants depend on the large-$h$ behavior of $P$. Finally, a simple asymptotic expansion of \eqref{EL} using \eqref{hp-p0-} shows that
\begin{equation}\label{ss-0}
h_{\textsc{ss}}(x)\sim
\left(\tfrac{A(m+1)^2}{2(m-1)}\right)^{\frac{1}{m+1}} (s-x)^{\frac{2}{m+1}},
\quad\mbox{as $x\to s^-$,}
\end{equation}
implying that the microscopic contact angle equals $\pi/2$ for any $m>1$ (Fig. \ref{fig:droplet}).

\medskip

\noindent \begin{minipage}[t]{1\textwidth}
\captionsetup{width=1\linewidth}
\centering\raisebox{\dimexpr \topskip-\height}{
  \includegraphics[width=\textwidth]{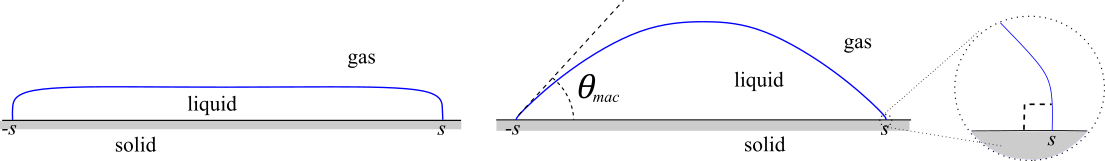}
  }
  \captionof{figure}{{\footnotesize
  Steady states under singular potentials: pancake ($S>0$, left) or droplet ($S<0$ and $e_*=+\infty$, right).}}
  \label{fig:droplet}
\end{minipage}

\medskip

If the model is assumed to hold down to $h=0$, the above characterization suffers from a limitation if $m\ge 3$. Indeed, it follows from \eqref{ss-0} that both summands in the energy, $h_x^2$ and $Q(h)$, are not integrable for $m\ge 3$. Therefore, steady states have unbounded energy if $m\ge 3$:
\begin{equation}
\label{limitation-1}
E[h_{\textsc{ss}}]=+\infty \quad\mbox{if \ $m\ge 3$.}
\end{equation}
The dual of this phenomenon is the following:
\begin{equation}\label{dual}
\begin{array}{l}\mbox{if $m\ge 3$
and $h$ has finite energy and positive mass, then $h$ can not tend to zero,}
\\
\mbox{either at any point or at infinity (cf. Lemma \ref{lem:repeat} in the Appendix).}
\end{array}
\end{equation}
This yields the variational counterpart of \eqref{limitation-1}:
\begin{equation}\label{limitation-2}
\mbox{mass-constrained minimizers of $E$ in $H^1(\R)$ do not exist if $m\ge 3$.}
\end{equation}

On the other hand, if the singularity is milder, that is if $1<m<3$, it was recently observed in \citet{DurG} that compactly supported minimizers $h_{min}$ do exist, and coincide with steady states with connected positivity set. In particular, $h_{min}$ satisfies \eqref{as-p} and \eqref{as-d}
with $e_*$ defined by \eqref{e*-DG}. In summary:
\begin{equation}\label{star}
\begin{array}{l}
\mbox{Steady states with connected positivity set exist}
\\
\mbox{for all $m>1$ and satisfy \eqref{as-p}, \eqref{as-d} and \eqref{ss-0}; however, }
\\
\mbox{their energy is finite and minimal if and only if $m<3$.}
\end{array}
\end{equation}

When $m\ge 3$, as for the van der Waals potential $P_0$, the limitation in \eqref{limitation-1}-\eqref{limitation-2} is usually handled by arguing that, at scales below a few molecules' radii (say, about ten Ångstr\"om for water), a continuum description of molecular interaction through $P_0$ may not be valid any more, since it is based on integration of binary molecular interactions. Hence the validity of a continuum description such as \eqref{TFE} is taken only up to a molecular threshold length-scale $\epsilon$ ($\epsilon^2=a^2=A/6\pi\gamma$ in \citet{DG}, \citet{gennes_hua_levinson_1990}, and \citet{LJ}).

\smallskip

On the other hand, \eqref{limitation-1}-\eqref{limitation-2} can not be ignored when one assumes $m\ge 3$ and seeks for a continuum model which consistently describes the liquid's profile {\em all the way down to $h=0$,} capturing ``pancake'' shapes in dry complete wetting ($S>0$). Unfortunately, simple fixes do not work. Indeed, either setting a virtual ``zero height'' at $\epsilon$ by the translation $\hat h =h-\epsilon$, or introducing a naive cut-off of the potential, such as $P(h)=\min\{P_0(h),P_{0}(\epsilon)\}$, make \eqref{TFE} non-singular; hence the final spreading equilibrium will not be a pancake either (mass-constrained minimizing sequences tend to zero: in more suggestive terms, the equilibrium is an unbounded layer of zero thickness). A much more ingenious fix dates back to the work of \citet{BP2} and consists in introducing a less singular ``molecular cut-off'', such as
\begin{equation}\label{P-model}
P(h)=\left\{\begin{array}{ll}
A_0 \epsilon^{m-3} h^{1-m} & \mbox{if $0<h\ll \epsilon$} \\[1ex]  P_0(h) & \mbox{if $\epsilon\ll h$}
\end{array}\right.
\quad 1< m<3,
\end{equation}
where $P_0$ is as in \eqref{vdw} and $\epsilon$ is a molecular-sized threshold length-scale.
Based on \eqref{star}, we expect that \eqref{P-model} may provide an equally effective description of droplets' profiles in the framework of van der Waals potentials, without the disadvantage of an infinite energy.

\smallskip

The goal of this manuscript is to explore, in the dynamical framework, a parallel between $m\ge 3$ and $m<3$ analogous to the one in \eqref{star}.
The preliminary matter in dynamical studies is of course that of travelling wave solutions, which will be discussed in \S \ref{s:tw}. In \S \ref{s:clc} we will identify a class of thermodynamically consistent contact-line conditions modelling contact-line friction, in the spirit of \citet{RE1,RHE,RE2}. Finally, in \S \ref{s:cop} we will draw our conclusions and present quite a few open questions. All of our observations will be supported by numerical examples.

\begin{rem}\label{rem:full}{\rm
An alternative approach to the contact-line paradox, which is quite common in the applied math community, is to take advantage of \eqref{dual}: when their validity down to $h=0^+$ is assumed, potentials which are sufficiently repulsive at $h=0^+$ and attractive at $\infty$, e.g. of the form
\begin{equation}\label{P-ra}
P(h)=B\left(h^{1-m} - h_*^{n-m}h^{1-n}\right), \quad B>0,\ \ m\ge 3, \ \ n<m
\end{equation}
yield periodic steady states which consist of arrays of droplets over a microscopic film of thickness $O(h_*)$ {\em fully} covering the substrate. Thus \eqref{P-ra} circumvents, rather than solving, the paradox. However, \eqref{P-ra} has proved to be particularly fruitful in numerical simulations and asymptotic studies, mainly in relation to dewetting phenomena: (in)stability of the flat film, bifurcation, concentration, and asymptotic scaling laws with respect to the potential's parameters,; see \cite{BGW,LP3,Getal,ORS,LW,DFHLHE}, the references therein, and \cite{W} for a recent overview. Potentials of the form \eqref{P-ra} have also been successfully employed in the analysis of the macroscopic dynamics of wetting: see e.g.  \citet{Eggers2005}, \citet{PE}, \citet{Savva2011}, and the references therein. However, this approach obviously can not capture pancake-shaped equilibria in dry complete wetting.
}\end{rem}

\begin{rem}\label{rem:pi}
{\rm
One may question whether a microscopic contact angle equal to $\pi/2$ is consistent with the small ratio of vertical vs horizontal length-scale required by lubrication approximation. In this respect, we should mention that the rigorous derivation of the thin-film equation in \citet{GO2} only requires a {\em global} smallness conditions on such ratio, in form of relations between mass, energy, and second moments (the result is proved for Darcy's flow in complete wetting, but it is plausible that similar conclusions may be drawn as well for Stokes flow and partial wetting). In fact, the validity of lubrication theory under global (hence weak) assumptions is most evident when looking at the statics: for instance, in partial wetting ($S<0$) with $P\equiv 0$ and one space dimension, it has been shown that
$$
\frac{\gamma}{\eps^2}\int_{\{h>0\}} \left((1+\eps^2 h_x^2)^{1/2}-1-\eps^2 S\right)\d x \ \stackrel{\eps\to 0}\to \   E[h] = \gamma\int_{\{h>0\}} \left(\tfrac12 h_x^2-S\right)\d x
$$
under the sole assumptions that $h\ge 0$ has finite energy, mass, and second moment, cf. \citet[(3) in Proposition 1]{GO1}. Note, however, that the finite-energy requirement is violated when $m\ge 3$. This is another indication of the necessity for a cut-off at a molecular-size length-scale $\epsilon$ when van der Waals forces dominate as $h\to 0$.
}\end{rem}

\section{Travelling waves}\label{s:tw}

As we mentioned, the idea that a film can spread because of a gradient of the disjoining pressure $\Pi=-P'$ is not new \citep{Starov,DG-CR,HDG}. In Section IV.C.3 of his review, \citet{DG} presents a heuristic analysis of advancing travelling waves for \eqref{TFE} in the case of van der Waals potentials.
We will now revisit part of his discussion in a more general way, i.e. assuming that
\begin{equation}\label{hp:p0}
P''(h)\sim A m h^{-m-1} \ \mbox{ as $h\to 0^+$,} \quad A>0, \quad m> 1, \quad P(0)=0
\end{equation}
and that
\begin{equation}\label{P-large}
P''(h)= B p h^{-1-p}(1+o(1)) \quad\mbox{as } \ h\to +\infty, \quad  p>1, \quad B\in \R.
\end{equation}
The cases $B>0$, resp. $B<0$, correspond to long-range repulsive, resp. attractive, potentials. Once again, for $m<3$ \eqref{hp:p0} may also be thought of as a cut-off of van der Waals potentials at molecular scales, such as in \eqref{P-model}. Since our focus is on the contact-line paradox, we adopt the no-slip condition, i.e. $\m(h)=\frac13 h^3$ (in this respect, see also Remark \ref{rem:n} in \S \ref{ss:global}).

\smallskip

We look for traveling-wave solutions to \eqref{TFE} with a constant speed $V$:
\begin{equation}\label{TW-ans}
h(t,x)=H(y), \quad y=x+ Vt, \quad V\in \R\setminus\{0\}.
\end{equation}
We require $H$ to display a contact line; capitalizing on translation invariance, we assume without losing generality that the contact line is located at $y=0$:
\begin{subequations}\label{TW-bc}
\begin{equation}
\label{TW-bc-0}
H(0)=0.
\end{equation}
In addition, we ask that $H$ connects to a bulk profile as $y\to +\infty$:
\begin{equation}\label{TW-bc-infty}
\supp H =[0,+\infty) \quad\mbox{and}\quad H(+\infty)=+\infty.
\end{equation}
\end{subequations}
Plugging the Ansatz \eqref{TW-ans} into \eqref{TFE} and using \eqref{def-Q} yields, in case of \eqref{h:m},
$$
3\mu V H_y + \gamma(H^3(H_{yy}-P'(H))_y)_y=0.
$$
Assuming null mass flux through the contact line gives, after an integration,
\begin{equation}
\label{TW}
U + H^2(H_{yy}-P'(H))_y=0, \qquad \mbox{where $U:= \frac{3\mu}{\gamma}V\in \R\setminus \{0\}$.}
\end{equation}
Note that $|U|$ coincides, up to a factor three, with the capillary number.
Finally, we ask $H$ to have {\em finite rate of bulk dissipation} near the contact line, in the sense that (cf. \eqref{fgf})
\begin{equation}\label{finite-dissipation}
\int_0^1 \frac{H^2}{\m(h)}V^2\d y  = \tfrac{\gamma^2}{3\mu^2}
\int_0^1 U^2 H^{-1} \d y <+\infty.
\end{equation}
For brevity, we will call {\em front} a solution to \eqref{TW}-\eqref{TW-bc} satisfying \eqref{finite-dissipation}.
Of course, as in the static case, it will be important to highlight which fronts have {\em finite energy} near the contact line, in the sense that
\begin{equation}\label{finite-energy}
\int_0^1 \left(\tfrac12 (H_y)^2 + Q(H)\right)\d y <+\infty.
\end{equation}

We will argue that singular potentials {\em generically} solve the contact-line paradox, in the following sense:

\medskip

\begin{boxlabel}
\item[{\bf (Q)}]
{\bf Quadratic fronts.} {\em Assume \eqref{def-Q}, \eqref{h:m}, \eqref{hp:p0}, and \eqref{P-large}. For any $U\in \R\setminus \{0\}$ there exists a {\em two parameter} ($a>0$, $b\in \R$), hence generic, family of fronts. They have quadratic growth as $y\to +\infty$,
\begin{equation}\label{as-i-1-1-res}
H(y)= ay^2+by + O(1) \quad\mbox{as $y\to +\infty$,}
\end{equation}
and satisfy
\begin{equation}
\label{as0}
\tfrac12 (H_y(y(H)))^2= P(H)\left(1 + c_1 H^{m-1} + o(H^{m-1})\right) \quad\mbox{as $H\to 0$,}
\end{equation}
where $O(1)$, $o(H^{m-1})$ and $c_1\in \R$ are determined by $m$, $U$, $P$, $a$, and $b$.
These fronts have finite energy if and only if $m<3$.
}
\end{boxlabel}

\medskip

\noindent This shows that, even if a no-slip condition is adopted at the contact-line, i.e. $\m(h)=\frac13 h^3$, singular potentials allow the existence of generic fronts (both advancing and receding) for any value of the normalized speed $U$. Since fronts have finite rate of dissipation,  this means that singular potentials stand as an alternative solution to the contact-line paradox. In terms of $H$, \eqref{as0} translates into
\begin{equation}\label{Hto0-res}
H(y) = \left(\tfrac{A (m+1)^2}{2(m-1)}\right)^{\frac{1}{m+1}}y^\frac{2}{m+1} (1+o(1)) \quad\mbox{as $y\to 0^+$,}
\end{equation}
which coincides with \eqref{ss-0}. Therefore,
as in the static case, fronts are compatible with a fully consistent continuum theory down to $h=0$ if $m<3$, whereas if $m\ge 3$ the energy is unbounded, and one must cut-off the fronts at a molecular length-scale $\epsilon$. In this respect, {\bf(Q)} parallels the static summary \eqref{star}.

\smallskip

Starting from the work of \citet{Voinov1976}, various formal asymptotic arguments have been developed for spreading droplets (see e.g. \citet{Greenspan,Hoc1,Cox,Ehrhard,Haley,Hoc2,BDDG,ES,Eggers2005,PE,CG}). In this framework, advancing fronts ($U>0$) are usually matched to a macroscopic profile. Such matching (parts of which were made rigourous in \citet{GO3,GGO,DelM}) requires to select those fronts which, instead of a quadratic one, display a linear (though logarithmically corrected) growth for $y\gg 1$; they are identified by
\begin{equation}
\label{H''0}
H_{yy}(y)\to 0 \quad \mbox{as $y\to +\infty$.}
\end{equation}
We will argue that, even if a no-slip condition is assumed at the contact-line, i.e. $\m(h)=\frac13 h^3$, such {\it linear-log} fronts also exist when singular potentials are adopted, for any positive speed:

\medskip

\begin{boxlabel}
\item[{\bf (L)}] {\bf Linear-log fronts.} {\em Assume \eqref{def-Q}, \eqref{h:m}, \eqref{hp:p0}, and \eqref{P-large}. For any $U>0$ there exists a {\em one-parameter} ($a\in \R$) family of
fronts $H_L$ such that \eqref{H''0} holds. They have linear-log growth in the sense that
\begin{equation}\label{as-i-2-2}
H_{y}^3(y(H)) = 3U\left(\log H -\tfrac13\log\log H +a + O(\tfrac{\log\log H}{\log H})\right) \quad \mbox{as $H\to +\infty$}
\end{equation}
and satisfy \eqref{as0}. These fronts have finite energy if and only if $m<3$.
}
\end{boxlabel}

\medskip

\noindent
Note that \eqref{as-i-2-2} does imply linear-log behavior as $y\to +\infty$:
\begin{equation}\label{as-i-2-1}
H_{y}^3(y)=3U\left(\log\left((3U)^{1/3} y\right)+ a +o(1)\right) \quad\mbox{as $y\to +\infty$}.
\end{equation}
As quadratic ones, also linear-log fronts are compatible with a fully consistent continuum theory down to $h=0$ only if $m<3$, whereas a molecular-sized cut-off is necessary if $m\ge 3$.

\begin{rem}\label{rem:n}{\rm
Since our focus is on the contact-line paradox, we have only discussed the mobility $\m(h)=\frac13 h^3$, corresponding to the no-slip condition. However, analogous arguments yield {\bf (Q)} and {\bf (L)} also for mobilities of the form $\m(h)=\frac13(h^3+\lambda^{3-n} h^n)$ with $n<3$.
}\end{rem}

In slippage models with $P\equiv 0$, a {\em single} linear-log front (and a {\em one-parameter} family of quadratic fronts) can be identified by imposing a condition on the value of the microscopic contact angle $H_y(0)$: in other words, the microscopic contact angle $H_y(0)$ may be taken as one of the parameters spanning the fronts. This additional condition on the contact angle is also necessary for uniqueness of generic solutions to \eqref{TFE} \citep{GKO,GGKO,K1,K2,KM1,KM2,Gn1,GP}. In {\bf (Q)} and {\bf (L)} above, the microscopic contact angle can not be a selection criterion, since all fronts have $H_y(0)=+\infty$. This points to the necessity of a different criterion which, for instance, singles out a unique linear-log front.

\smallskip

One possible selection criterion is the so-called {\it maximal film}, an advancing travelling wave supported in $\R$ (i.e., without a contact line). For van der Waals potentials, the maximal film has been successfully employed as a first approximation of droplets' advancing fronts for large positive spreading coefficient, where a prominent precursor region is supposed to form ahead of the macroscopic contact line. For instance, it is used by \citet{HDG} and \citet{DG} to infer the Voinov-Cox-Hocking logarithmic correction to Tanner's law \citep{Voinov1976,tanner,Cox,Hoc1,Hoc2}; see e.g. the discussion in \citet{ES}. We will argue that such maximal film exists for any $m\ge 2$:

\medskip

\begin{boxlabel}
\item[{\bf (M)}] {\bf Maximal film.} {\em Assume \eqref{def-Q}, \eqref{h:m}, \eqref{hp:p0}, and \eqref{P-large}. For any $U>0$ and $m\ge 2$, there exists a unique (up to translations) solution $H_M$ to \eqref{TW} in $\R$ such that $H_M(y)\to 0^+$ as $y\to -\infty$ and \eqref{H''0} holds. They satisfy \eqref{as-i-2-1} and
\begin{equation}\label{H-sep}
\begin{array}{ll}
H_{M}(y)\sim \left(-\tfrac{(m-2)U}{A m}y\right)^{\frac{1}{2-m}} & \mbox{ if  $m>2$}
\\[2ex]
H_{M}(y)\sim e^{\frac{U}{2A}y} & \mbox{ if  $m=2$}
\end{array}
\quad\mbox{as $y\to -\infty$},  \ \ U>0.
\end{equation}
}
\end{boxlabel}

In \S \ref{ss:as-cl}-\ref{ss:global} we provide the formal asymptotic arguments which motivate {\bf (Q)}, {\bf (L)} and {\bf (M)}; in \S \ref{ss:num} we give numerical examples supporting them; finally, in \S \ref{ss:slip} we compare them with analogous results for the slippage model.

\subsection{Asymptotics near the contact line}\label{ss:as-cl}

Note that \eqref{TW} is autonomous: as customary, it is convenient to get rid of translation invariance by exchanging dependent and independent variable, thus reducing the order of the ODE. Therefore, we let
\begin{equation}
\label{sost}
\psi(H)=\tfrac12 H^2_y(y(H)).
\end{equation}
For $H_y>0$ in a neighborhood of $H=0$, \eqref{TW} reads as
\begin{equation}
\label{TW-in2}
H^2 \psi''(H)=-\frac{U}{\sqrt{2\psi(H)}} + H^2 P''(H).
\end{equation}
At leading order as $H\to 0$, a simple asymptotic expansion using \eqref{hp:p0} shows that two cases occur:
\begin{equation}
\label{ab}
(a)\ \ \psi(H)\sim P(H);
\qquad (b)\ \ \psi(H) \sim \tfrac12\left(\tfrac{U}{H^2P''(H)}\right)^2
, \ U>0.
\end{equation}
Case $(a)$, resp. $(b)$, may be read off from \eqref{TW-in2} by neglecting the first term on the right-hand side, resp. the left-hand side.

\smallskip

We anticipate that the solutions in $(a)$ are generic and have finite rate of dissipation, whereas those in $(b)$ are non-generic and have unbounded rate of dissipation. However, the solutions in $(b)$ will capture the maximal film identified in {\bf (M)}. We thus distinguish the two cases.

\subsubsection{Case (a)}

We linearize \eqref{TW-in2} around $P(H)$: define the function $v(H)$ as
$$
\psi(H)=P(H)(1+ v(H)), \quad v(0)\stackrel{\eqref{hp:p0},(\ref{ab})}= 0.
$$
It follows from \eqref{TW-in2} that
\begin{align}\label{ancora}
L(v(H)) = -\tfrac{U}{\sqrt{2P^3(H)}}(1+v(H))^{-\frac{1}{2}},
\end{align}
where
\begin{align*}
L(v(H)) & = H^2 \left(v''(H)+2\tfrac{P'(H)}{P(H)} v'(H) + \tfrac{P''(H)}{P(H)}v(H)\right).
\end{align*}
The linearization of \eqref{ancora} around $v=0$ is given by
\begin{align*}
& H^2 v''(H) - 2(m-1)Hv'(H) + m(m-1) v(H)
\\ & \qquad = -\underbrace{\tfrac{U}{\sqrt{2P^3(H)}}} -2\underbrace{\left(\tfrac{HP'(H)}{P(H)}+(m-1)\right)}Hv'(H)+ \underbrace{\left(m(m-1)-\tfrac{H^2P''(H)}{P(H)}\right)}v(H).
\end{align*}
The left-hand side is an Euler equation, whereas the bracketed coefficients on the right-hand side are $o(1)$ as $H\to 0^+$ in view of \eqref{hp:p0}. Therefore the equation has a two-parameter family of solutions of the form
$$
v(H)=c_1 H^{m-1} + c_2 H^m +v_p(H), \quad c_1,c_2\in \R.
$$
Generically, $v_p(H)=o(H^{m-1})$ as $H\to 0^+$ (if $c_1=0$, its regularity improves) is a function determined by $U$, $P$, $m$, $c_1$, and $c_2$. In terms of $H$, this translates into \eqref{as0}.
Returning to the $y$ variable, the additional degree of freedom coming from invariance under translation $y\mapsto y-y_0$ is spent to match the condition $H(0)=0$. Therefore \eqref{as0} translates into
\begin{equation}\label{Hto0}
H(y) = \left(\tfrac{A (m+1)^2}{2(m-1)}\right)^{\frac{1}{m+1}}y^\frac{2}{m+1} (1+o(1)) \quad\mbox{as $y\to 0^+$.}
\end{equation}
Since $\frac{2}{m+1}<1$ for $m>1$, these solutions have finite rate of dissipation; in addition, they have finite energy if and only if $m<3$:
\begin{align*}
\int_0^1\left(\tfrac12 H_y^2+Q(H)\right)\d y \stackrel{\eqref{Hto0}, \eqref{hp:p0}}\approx &
\int_0^1 y^{\frac{2(1-m)}{m+1}}\d y <+\infty \quad\Leftrightarrow \quad m<3.
\end{align*}
In summary:

\medskip

\begin{boxlabel}
\item[{\bf (TW$_0$)}] {\em  Assume \eqref{def-Q}, \eqref{h:m}, and \eqref{hp:p0}.
Locally for $y\ll 1$, for any $U\in \R\setminus\{0\}$ there exists a two-parameter family of generic solutions $H$ to \eqref{TW}-\eqref{TW-bc-0} satisfying \eqref{as0} and \eqref{Hto0-res}.
Their rate of bulk dissipation is finite, in the sense that \eqref{finite-dissipation} holds; if $m<3$ their energy is also finite, in the sense that \eqref{finite-energy} holds.
}
\end{boxlabel}

\medskip

\noindent
We note for further reference that the traveling waves in {\bf (TW$_0$)} satisfy
\begin{align}
\label{as1b}
H_{yy}-P'(H) \stackrel{\eqref{sost}}= \psi'(H) -P'(H) \stackrel{\eqref{hp:p0},\eqref{as0}} =o(H^{-1}) \quad\mbox{as $H\to 0^+$}.
\end{align}

\subsubsection{Case (b)}

We recall that $U>0$ in this case. We linearize \eqref{TW-in2} around $\psi_0(H)= \frac12 \left(\frac{U}{H^2P''(H)}\right)^2$. In fact, it is convenient to write the linearization in terms of $\psi$ itself, writing $\psi=\psi_0\left(1+\left(\frac{\psi}{\psi_0}-1\right)\right)$:
$$
H^2 \psi''(H)= H^2P''(H) -\tfrac{U}{\sqrt{2\psi_0}}\left(1-\tfrac12\left(\tfrac{\psi}{\psi_0}-1\right)\right) = \tfrac12 H^2 P''(H)\left(\tfrac{\psi}{\psi_0}-1\right).
$$
In view of the asymptotic of $P$ in \eqref{hp:p0}, this means that
$$
\psi''(H) =C^2H^{1-3m}\psi(H)(1+o(1)) -\tfrac{Am}{2}H^{-m-1}(1+o(1)), \quad C=\sqrt{\tfrac{A^3m^3}{U^2}}, \quad \mbox{as $H\to 0^+$}.
$$
At leading order as $H\to 0^+$, the change of variables
$$
\psi(H)=H^{\frac12} \hat\psi(\eta), \quad \eta= DH^\beta, \quad \beta=\tfrac32(1-m)<0, \quad D=\tfrac{C}{|\beta|},
$$
leads to a non-homogeneous modified Bessel equation,
$$
\eta^2\hat\psi''(\eta) + \eta \hat\psi'(\eta) -\left(\eta^2+(2\beta)^{-2}\right)\hat\psi(\eta)=-\tfrac{Am}{2\beta^2}\left(\tfrac{|\beta|}{C}\eta\right)^{\frac{2m-1}{3(m-1)}}  \quad \mbox{as $\eta\to +\infty$}.
$$
The homogeneous solutions are spanned by modified Bessel functions \citep{bessel}: $\hat\psi(\eta)= c_1 K(\eta) + c_2 I(\eta)$, where $K=K_{(2|\beta|)^{-1}}$ and $I=I_{(2|\beta|)^{-1}}$. Since $I$ is unbounded as $\eta\to +\infty$, the condition $\psi(0)=0$ implies that $c_2=0$. Simple computations thus show that $\hat\psi(\eta)= c_1 K(\eta) + \hat\psi_p(\eta)$ with $c_1\in \R$, where the particular solution $\hat\psi_p$ is given by
$$
\hat\psi_p(\eta)= \tfrac{Am}{2\beta^2} D^{\frac{1-2m}{3(m-1)}}\left(K(\eta)\int (\eta')^{\frac{2m-1}{3(m-1)}-1}I(\eta')\d \eta' -I(\eta) \int (\eta')^{\frac{2m-1}{3(m-1)}-1}K(\eta')\d \eta'\right)
$$
(here we used that the Wronskian $KI'-IK'=\eta^{-1}$). A simple asymptotic expansion, using $I(\eta)\sim \frac{1}{\sqrt{2\pi}}\eta^{-1/2}e^\eta$ and $K(\eta)\sim \frac{\sqrt{\pi}}{\sqrt{2}}\eta^{-1/2}e^{-\eta}$ as $\eta\to+\infty$, shows that
$$
\hat\psi_p(\eta)\sim \frac{Am}{2\beta^2D^2} \left(\frac{\eta}{D}\right)^{\frac{5-4m}{3(m-1)}}\quad \mbox{as $\eta\to+\infty$}.
$$
Returning to $\psi(H)=\frac12 H_y^2(y(H))$, this means that
\begin{subequations}\label{as-zero}
\begin{equation}\label{as-zero-1}
\tfrac12 H_y^2(y(H)) = \psi_p(H) + \psi_o(H) \sim \psi_p(H) \quad\mbox{as $H\to 0^+$}, \quad c_1\in \R,
\end{equation}
where $\psi_o$ and $\psi_p$ are functions depending on $m$, $U$, and $P$; they satisfy
\begin{equation}\label{as-zero-2}
\psi_p(H) \sim \tfrac{U^2}{2A^2m^2}H^{2(m-1)}\quad\mbox{and}\quad \psi_o(H) \approx H^{\frac{3m-1}{4}}\exp\left(-\tfrac{Am\sqrt{Am}}{U|\beta|}H^{\frac32(1-m)}\right) \quad\mbox{as $H\to 0^+$}.
\end{equation}
\end{subequations}
In terms of $H(y)$ as $y\to 0^+$, for $U>0$ and $m<2$ \eqref{as-zero} translates into a one-parameter, hence non-generic, family of solutions to \eqref{TW}-\eqref{TW-bc-0} which had already been identified by \cite{BP2}: they all behave as
$$
H(y)\sim \left(\tfrac{2-m}{Am}Uy\right)^{\frac{1}{2-m}}, \ U>0, \ m<2.
$$
Note the constraint $m<2$: if $m\ge 2$, $H$ diverges as $y\to 0^+$, hence it is not admissible. It is clear that these solutions have unbounded rate of dissipation, since $\frac{1}{2-m}>1$ for $m\in (1,2)$ (cf. \eqref{finite-dissipation}).

\smallskip

For $m\ge 2$, the above argument also reveals a two-parameter (including translation, whence non-generic) family of separatrices, which may be identified by requiring that $H(y)\to 0^+$ as $y\to -\infty$. Therefore:

\medskip

\begin{boxlabel}
\item[{\bf (TW$_{-\infty}$)}] {\em  Assume \eqref{def-Q}, \eqref{h:m}, and \eqref{hp:p0}. For $U>0$ and $m\ge 2$, there exists a two-parameter (including translation) family of solutions $H_{\rm{sep}}$ to \eqref{TW} satisfying \eqref{H-sep}.
}
\end{boxlabel}

\subsection{Asymptotics in the liquid bulk}\label{ss:as-far}

We now look at the behavior of solutions to \eqref{TW} such that $H(y)\to +\infty$ as $y\to +\infty$. We will argue that:

\medskip

\begin{boxlabel} {\em
\item[{\bf (TW$_\infty$)}] Consider \eqref{TW} with \eqref{TW-bc-infty}. Assume \eqref{def-Q}, \eqref{h:m}, and \eqref{P-large}.
\smallskip
\begin{boxlabeltwo}
\item[{\bf (TW$_\infty$-Q)}] For any $U\in \R\setminus\{0\}$ there exists a generic, three-parameter (including translation) family of quadratically growing solutions:
\begin{equation}\label{as-i-1-1}
H(y)= a(y-y_0)^2+b(y-y_0) + O(y^{-\gamma}) \quad\mbox{as $y\to +\infty$,}  \quad \gamma=\min\{1,2p-2\},
\end{equation}
with $a>0$ and  $b,y_0\in\R$.
\item[{\bf (TW$_\infty$-L)}] For any $U>0$ there exists a non-generic, two-parameter (including translation) family of {\em linear-log solutions} satisfying \eqref{as-i-2-2} with $a\in \R$.
\end{boxlabeltwo}
}\end{boxlabel}

\medskip

To motivate {\bf (TW$_\infty$)}, in view of \eqref{P-large} we rewrite \eqref{TW} as
\begin{equation}\label{far}
H_{yyy} =-UH^{-2} +pBH^{-p-1}H_{y}(1+r(H)),\quad r(H)=o(1) \mbox{ \ as $H\to +\infty$}.
\end{equation}
The asymptotic expansion yielding {\bf (TW$_\infty$-Q)} is straightforward. For {\bf (TW$_\infty$-L)}, let $U>0$. The equation for
$$
u(H)=\frac{1}{3U} H_{y}^3(y(H))
$$
is
$$
u''=\frac{(u')^2}{3u}- \frac{1}{H^2} +3pBH^{-p-1}(3U)^{-2/3}u^{1/3}(1+r(H)),\quad\mbox{$r(H)=o(1)$ as $H\to +\infty$.}
$$
Following \citet{GGO}, we exploit the homogeneity of the $(B=0)$-part of the equation letting
$s=\log H$, which yields
\begin{equation}
\label{fgh}
\frac{\d^2 u}{\d s^2}= \frac{\d u}{\d s} + \frac{1}{3u}\left(\frac{\d u}{\d s}\right)^2- 1 + f(s,u), \quad f(s,u)=3pB (3U)^{-2/3} e^{s(1-p)}u^{1/3}(1+r(s)).
\end{equation}
This equation has been analysed in \citet[Section 4 and 5]{GGO}, with a slippage-type perturbation (namely, $f(s,u)= (1+e^{(3-n)s})^{-1}$) whose specific form is however immaterial as long as $f(s)=O(s^{-2}\log s)$. Their analysis shows that \eqref{fgh} with $f=0$ has a one-parameter family of solutions such that $u(s)/s\to 1$ and $u'(s)\to 1$ as $s\to +\infty$, with an asymptotic expansion of the form
$$
u(s)= \left(s-\tfrac13 \log s +a + O(s^{-1}\log s)\right) \quad\mbox{for $s\gg 1$}, \quad a\in \R.
$$
Since $f(s,u)\approx e^{s(1-p)}u^{1/3}$ and $p>1$, it is apparent that $f$ produces only an exponentially small perturbation: thus solutions to \eqref{fgh} have the same behavior, which in terms of $H=e^s$ yields \eqref{as-i-2-2}.

\subsection{Global behavior}\label{ss:global}

For the global picture, one has to make sure that local solutions are global. This is not always the case, in the sense that \eqref{TW} also has generic solutions with compact support, a feature which is common to the slippage model. However, we have strong numerical evidence that generic members of both {\bf (TW$_\infty$-Q)} and {\bf (TW$_\infty$-L)} touch down to $H=0$ at some point $y_0\in \R$ (see \S \ref{ss:num}). Capitalizing on translation invariance, one of the three parameters in {\bf (TW$_\infty$-Q)} may be used to match $H(0)=0$, yielding a two-parameter family of fronts satisfying {\bf (TW$_0$)}: this yields {\bf (Q)}.
Analogously, one of the two parameters in {\bf (TW$_\infty$-L)} may be used to match $H(0)=0$, yielding a one-parameter family of fronts satisfying {\bf (TW$_0$)}: this yields {\bf (L)}.
Finally, up to a translation, there exists a one parameter family of separatrices $H_{\rm sep}$ emanating from $-\infty$ as well as a one-parameter family of linear-log fronts emanating from $+\infty$. Since \eqref{TW} is of third order but autonomous, this entails a unique (up to translation) {\it maximal film} $H_M$ which satisfies \eqref{TW} and is such that both {\bf (TW$_{-\infty}$)} and {\bf (TW$_\infty$-L)} hold: this yields {\bf (M)}.

\subsection{Numerical observations}\label{ss:num}

In order to provide numerical evidence of {\bf(Q)}, {\bf(L)}, and {\bf(M)}, we take as prototype example
\begin{equation}\label{Q-proto}
Q(h)= \tfrac{A}{m-1}h^{1-m}-S.
\end{equation}
The reason for considering such simple model instead of, for instance, \eqref{P-model} or \eqref{P-ra}, is twofold.
Firstly, this choice is sufficient for a first numerical check on {\bf(Q)}, {\bf(L)} and {\bf(M)}, since they quantitatively depend only on the behavior of $P$ as $h \to 0^+$ (its sign and decay at infinity matters only qualitatively, independently of the power-law exponent $p$). Secondly, we can capitalize on the homogeneity of $P$ to normalize the dimensionless speed $U$ to $\pm 1$: indeed, letting
\begin{equation}
\label{scaling-m2}
H=A^\frac{1}{m-1} |U|^{-\frac{2}{3(m-1)}}\hat H, \quad y=A^\frac{1}{m-1} |U|^{-\frac{m+1}{3(m-1)}} \hat y,
\end{equation}
we may rewrite \eqref{TW} with \eqref{Q-proto} as
\begin{equation}\label{TW-m2-1}
\hat H^2(\hat H_{\hat y\hat y}+\hat H^{-m})_{\hat y}=-\frac{U}{|U|}.
\end{equation}
Of course, a quantitative investigation of the fronts' behavior in the intermediate regions and/or in terms of the parameters, will require both a more careful choice of $P$ and a more extensive numerical study, both outside the scope of this contribution. Note that a change in sign of $U$ is equivalent, in \eqref{TW-m2-1}, to a change of sign of $\hat y$, and that the left-hand side of \eqref{bc-RE-m2} is unaffacted by the latter change: hence, removing hats,  in place of \eqref{TW-m2-1} we will equivalently consider
\begin{equation}\label{TW-m2-2}
H^2(H_{yy}+H^{-m})_{y}=-1,
\end{equation}
with the understanding that

\smallskip

\begin{itemize}
\item advancing fronts ($U>0$) correspond to solutions to \eqref{TW-m2-2} with $H(0)=0$, $\supp H=[0,+\infty)$ and $H(+\infty)=+\infty$;
\smallskip
\item receding fronts ($U<0$) correspond to solutions to \eqref{TW-m2-2} with $H(0)=0$, $\supp H=(-\infty,0]$ and $H(-\infty)=+\infty$.
\end{itemize}

\smallskip

Generic solutions to \eqref{TW-m2-1} can be obtained by noting that $H$ is concave near $H=0$ (see \eqref{Hto0}) and convex near $H=+\infty$ (see \eqref{as-i-1-1} and \eqref{as-i-2-1}). Therefore there exists a point $y$ such that $H_{yy}=0$. By translation invariance, we may fix that point to be $y=1$. Shooting from $y=1$ with the two parameters $H(1)$ and $H_{y}(1)$ produces a two-parameter family of solutions, which can then be translated in $y$ to match $H(0)=0$. Consider therefore
\begin{equation}\label{qw}
1+H^2(H_{yy}+ H^{-m})_{y}=0, \quad H(1)=\alpha>0, \quad H_{y}(1)=\beta\in\R, \quad H_{yy}(1)=0.
\end{equation}
For a fixed $\alpha$, the generic picture is the following (cf. Fig. \ref{generic-n3-m2} and Fig. \ref{generic-n3-m3}):

\smallskip
\begin{itemize}
\item advancing, quadratic fronts for $\beta>\beta_0(\alpha)$;
\smallskip
\item an advancing, linear-log front $H_L$ for $\beta=\beta_0(\alpha)$;
\smallskip
\item compactly supported solutions for $\beta\in (\beta_1(\alpha),\beta_0(\alpha))$;
\smallskip
\item a separatrix $H_{\rm{sep}}$ for $\beta=\beta_1(\alpha)$;
\smallskip
\item receding (quadratic) fronts for $\beta<\beta_1(\alpha)$.
\end{itemize}

\smallskip

\noindent The picture confirms the existence of both a two-parameter family of quadratic fronts (both advancing and receding, see {\bf (Q)}) and two one-parameter families of advancing linear-log fronts (see {\bf (L)}), resp. separatrices (see {\bf (TW$_{-\infty}$)}). In this case, the two parameters which span the fronts are taken to be $\alpha=H|_{H_{yy}=0}$ and $\beta=H_y|_{H_{yy}=0}$.

\begin{minipage}[t]{1\textwidth}
\captionsetup{width=1\linewidth}
\centering\raisebox{\dimexpr \topskip-\height}{
  \includegraphics[width=0.55\textwidth]{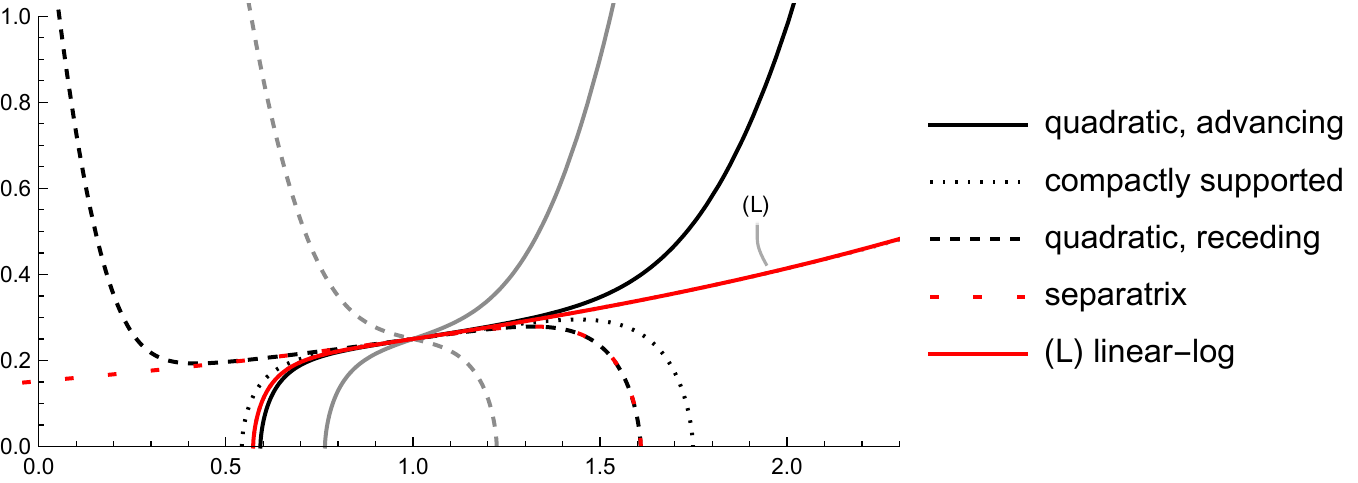}
\quad  \includegraphics[width=0.35\textwidth]{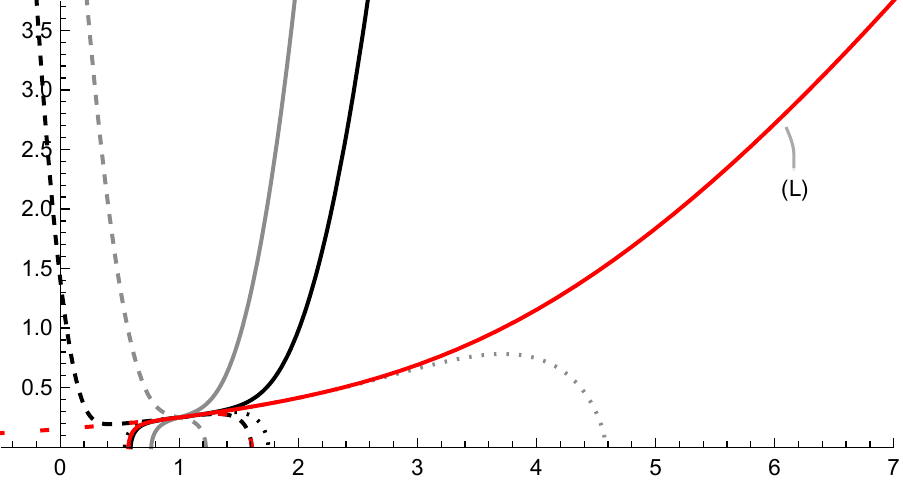}
  }
\vskip 2ex
  \centering\raisebox{\dimexpr \topskip-\height}{
  \includegraphics[width=0.55\textwidth]{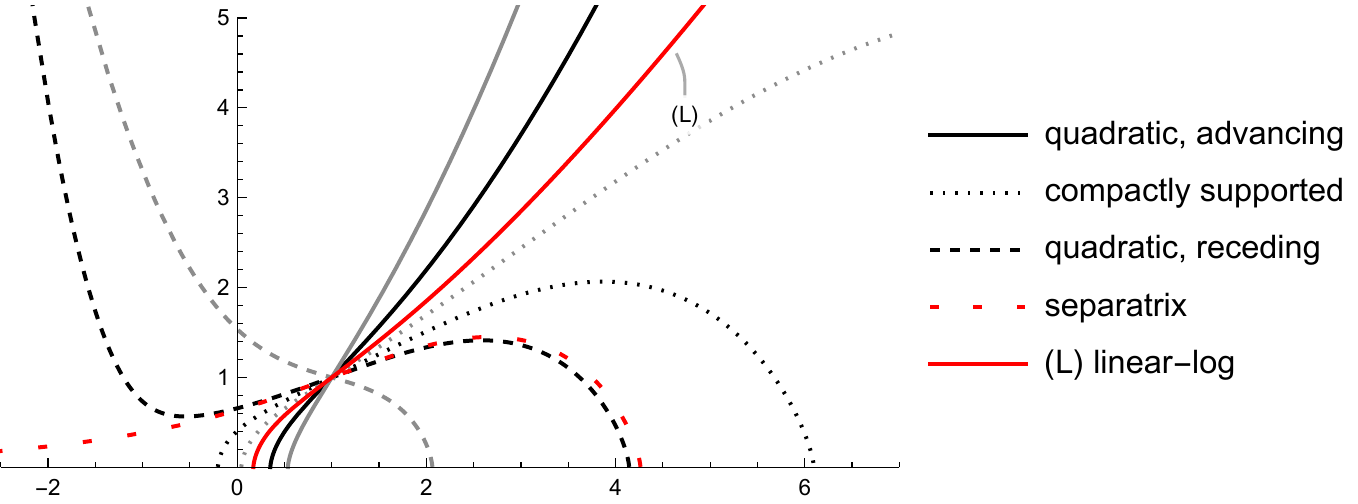}
  \includegraphics[width=0.35\textwidth]{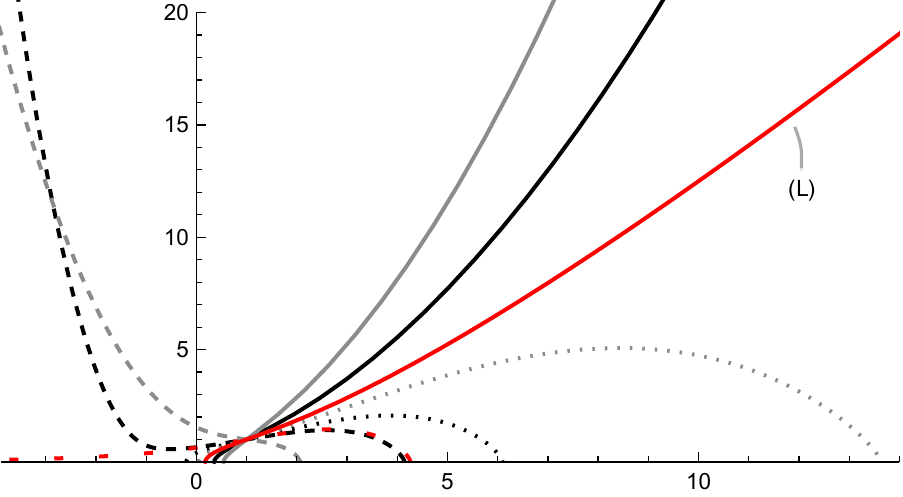}
  }
  \captionof{figure}{{\footnotesize Solutions to \eqref{qw} with $m=2$, $\alpha=1/4$ (top) and $\alpha=1$ (bottom), at two different scales. For $\alpha=1/4$ (top), the separatrix $H_{\rm{sep}}$ and the black receding front are indistinguishable for small heights, as well as the gray compactly supported solution and the linear-log front. }}
  \label{generic-n3-m2}
\end{minipage}

$\ $

\begin{minipage}[t]{1\textwidth}
\captionsetup{width=1\linewidth}
\centering\raisebox{\dimexpr \topskip-\height}{
  \includegraphics[width=0.55\textwidth]{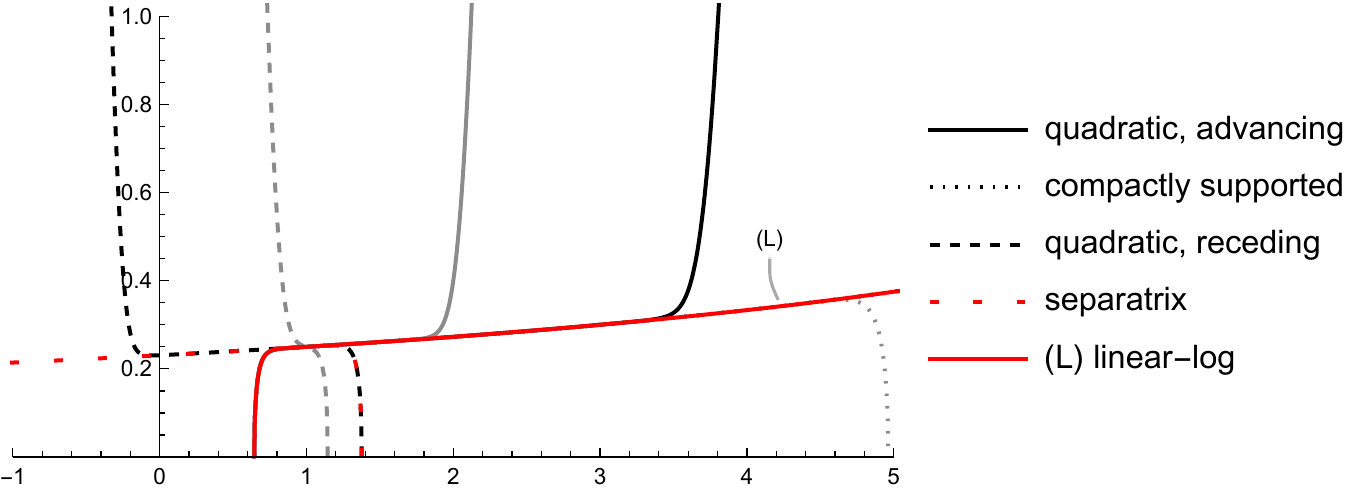}
\quad  \includegraphics[width=0.35\textwidth]{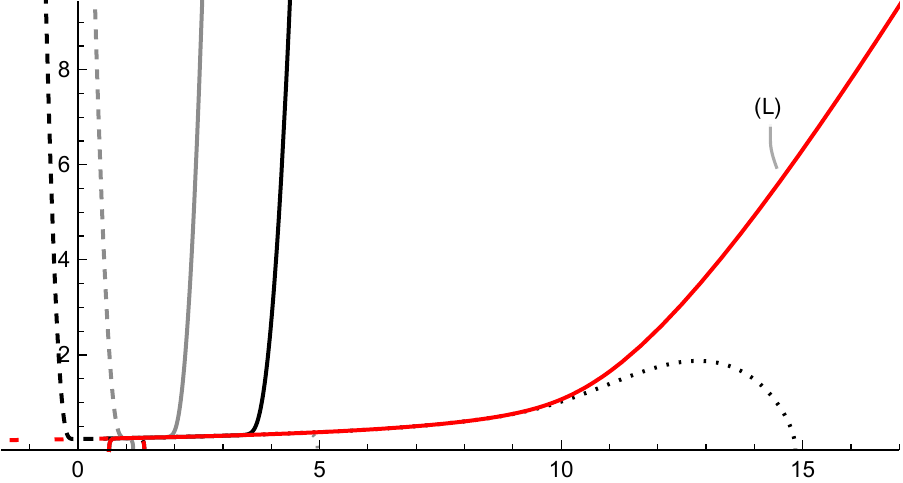}
  }
\vskip 2ex
  \centering\raisebox{\dimexpr \topskip-\height}{
  \includegraphics[width=0.55\textwidth]{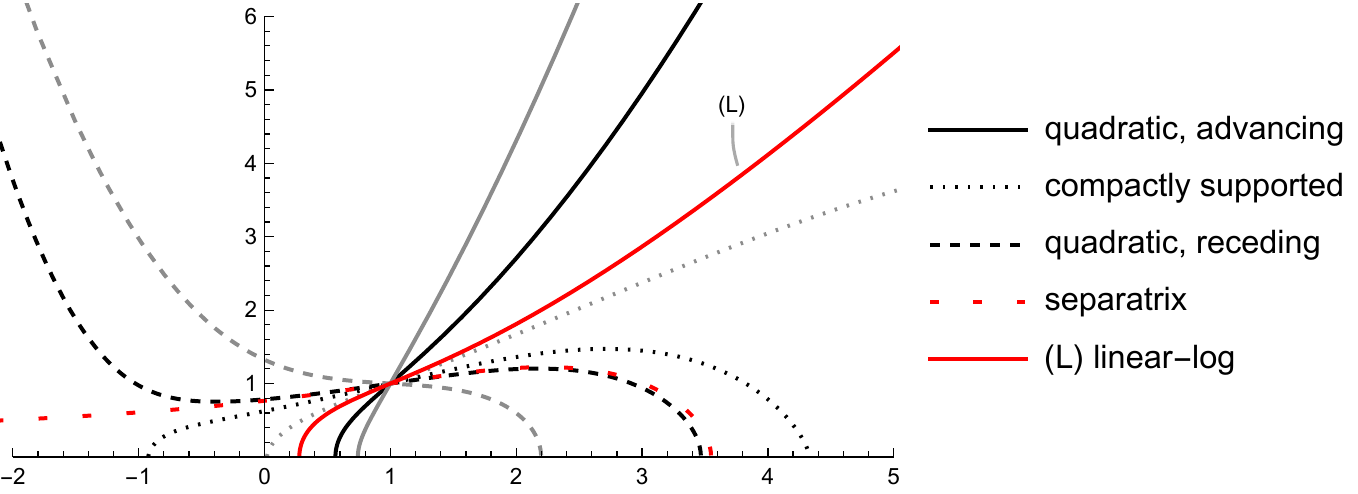}
\quad  \includegraphics[width=0.35\textwidth]{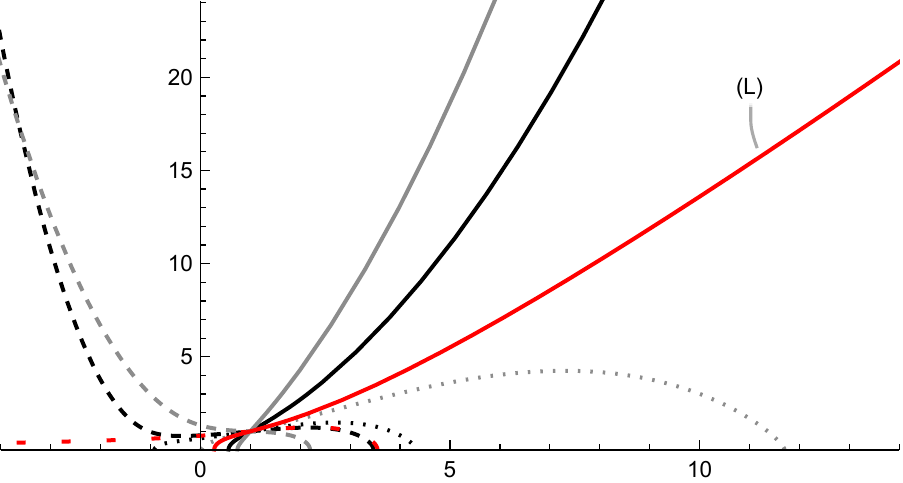}
  }
  \captionof{figure}{\footnotesize Solutions to \eqref{qw} with $m=3$, $\alpha=1/4$ (top) and $\alpha=1$ (bottom), at two different scales. For $\alpha=1/4$ (top) quadratic fronts, linear-log front, and compactly supported solutions are indistinguishable for small heights, as well as the separatrix and the black receding front. }
  \label{generic-n3-m3}
\end{minipage}

$\ $

It is interesting to compare the shapes of the linear-log fronts $H_L$ for varying values of $\alpha=H_L|_{(H_L)_{yy}=0}$. The shapes reported in Fig. \ref{tanner-n3-m2} show that $H_L$ increase as $\alpha$ increases. It also shows that a prominent precursor region forms ahead of the ``macroscopic contact line'' for small values of $\alpha$.

\medskip

\noindent \begin{minipage}[t]{1\textwidth}
\captionsetup{width=1\linewidth}
\centering\raisebox{\dimexpr \topskip-\height}{
  \includegraphics[width=0.4\textwidth]{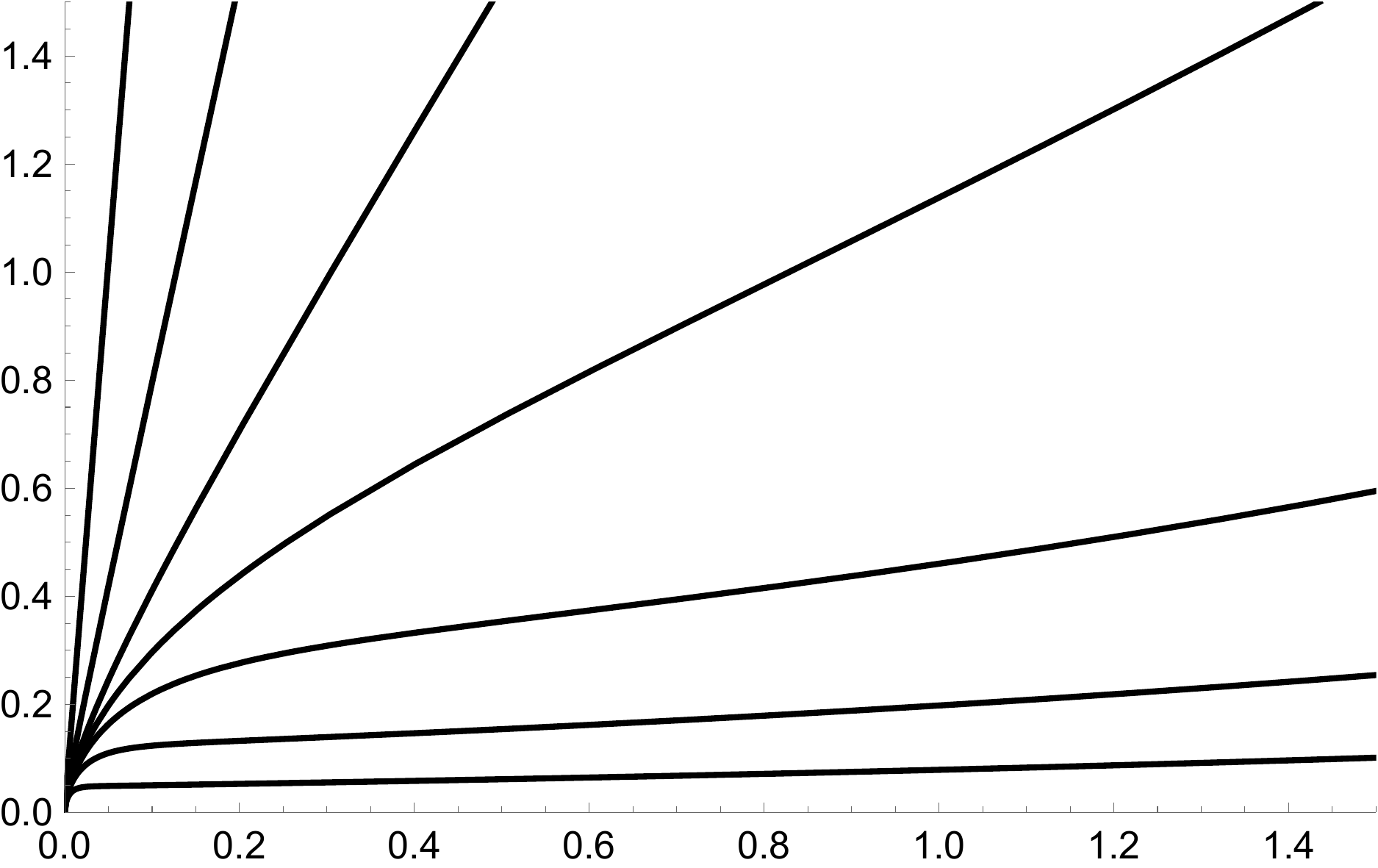}\qquad
  \includegraphics[width=0.4\textwidth]{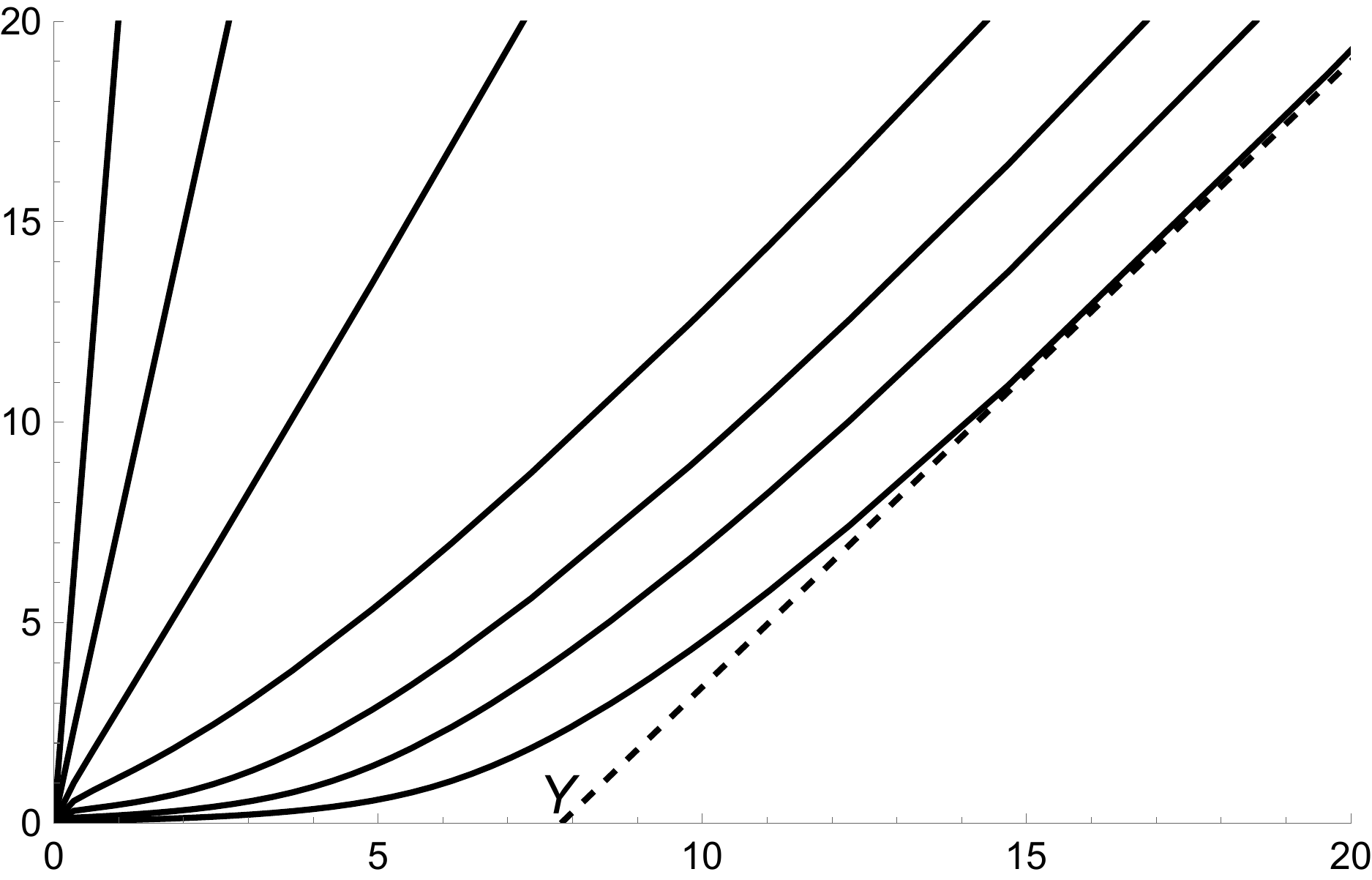}
  }
  \captionof{figure}{{\footnotesize Linear-log solutions to \eqref{qw} with $m=2$ for $\alpha=H_L|_{(H_L)_{yy}=0}=e^k$, with $k$ ranging from $-3$ (bottom) to $3$ (top), at two different scales. For $k=-3$, on the right, the ``macroscopic contact line'' $Y$.}}
  \label{tanner-n3-m2}
\end{minipage}

\medskip

Figure \ref{fig:maximal} shows the maximal film $H_M$ (see {\bf (M)}), which is obtained observing that $(H_M)_{yy}(y)\to 0$ as $y\to \pm\infty$: hence there exists $y_0\in \R$ such that $(H_M)_{yyy}(y_0)=0$, which implies that $m(H_M)_y(y_0)=H_M^{m-1}(y_0)$. Then $H_M$ is identified by shooting from $y_0$, using $H_M(y_0)$ and $(H_M)_{yy}(y_0)$ as parameters.

\medskip

\begin{minipage}[t]{1\textwidth}
\captionsetup{width=1\linewidth}
\centering\raisebox{\dimexpr \topskip-\height}{
  \includegraphics[width=0.45\textwidth]{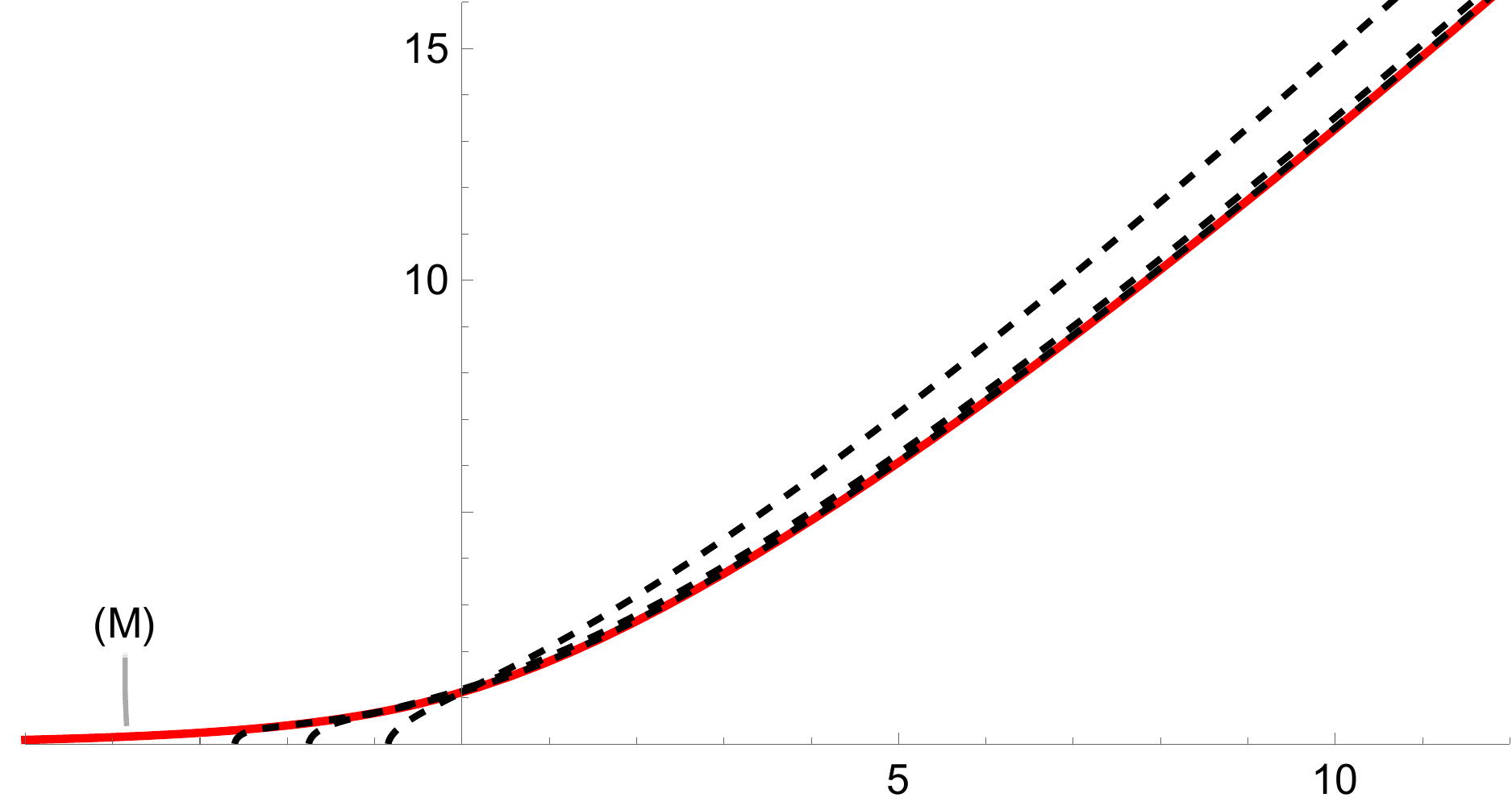}
\quad  \includegraphics[width=0.45\textwidth]{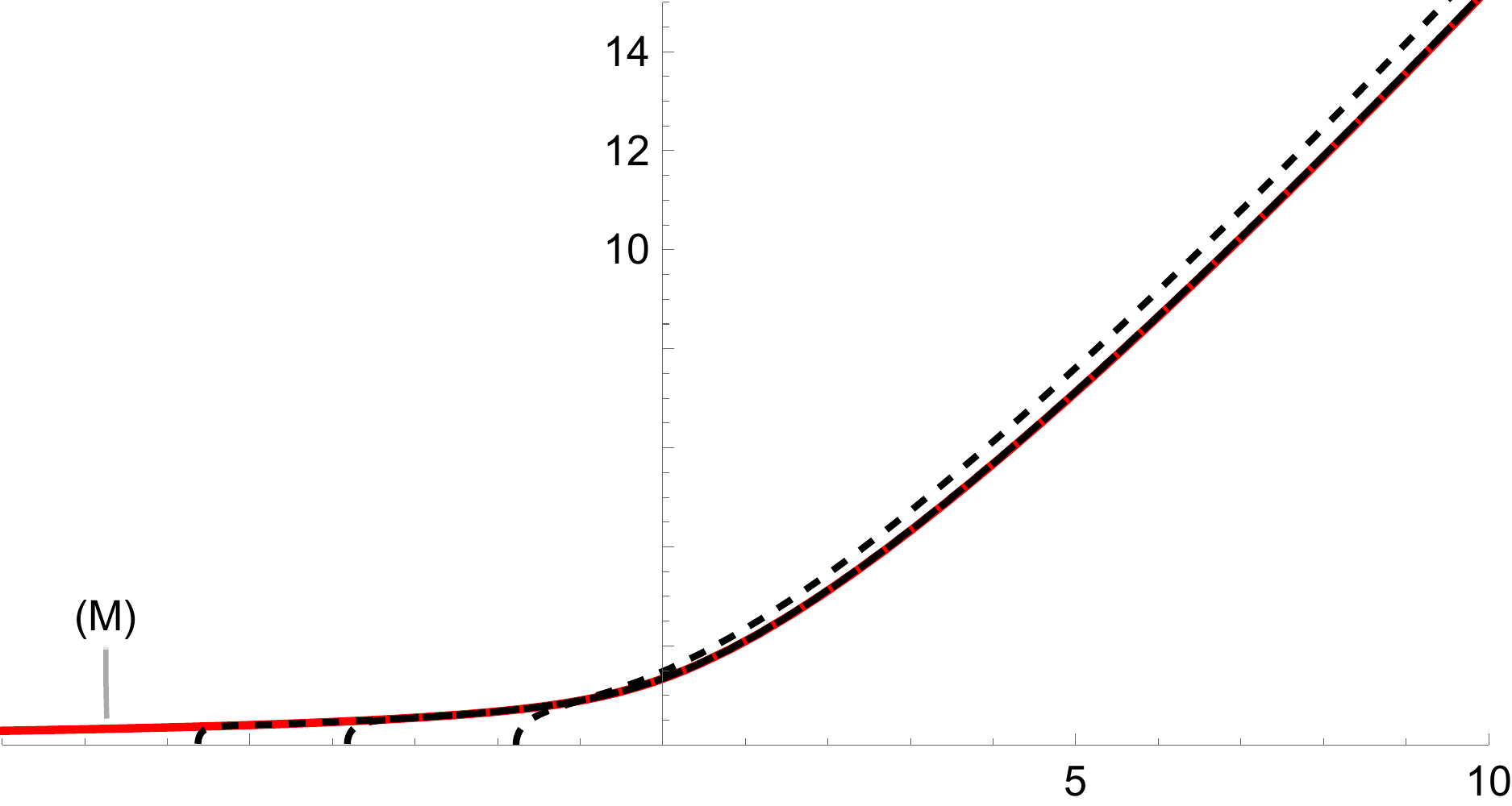}
  }
  \captionof{figure}{{\footnotesize The maximal film (solid) for $m=2$ (left) and $m=3$ (right), together with a few linear-log fronts (dashed).}}
  \label{fig:maximal}
\end{minipage}

\subsection{Comparison with slippage models}\label{ss:slip}

It is useful to compare the features in {\bf(Q)}, {\bf (L)} and {\bf (M)} with parallel ones for the case $P\equiv 0$ under slip conditions. In this case, traveling wave solutions (if they exist) solve
\begin{equation}
\label{TW-slip}
U + (H^2+ \lambda^{3-n} H^{n-1}) H_{yyy}=0
\end{equation}
with the same boundary conditions.

\smallskip

First of all, {\bf (Q)} and {\bf (L)} obviously contrasts \eqref{TW-slip} in the no-slip case $\lambda =0$. Indeed, solutions to \eqref{TW-slip} with $\lambda=0$  exist only if $U<0$ (receding), but their rate of bulk dissipation density is not integrable near $y=0$: indeed,
\begin{equation}\label{tw-noslip}
H(y)\sim \left(-\tfrac{9\mu}{\gamma}V\right)^{1/3} y\log^{1/3}\tfrac1y, \quad\mbox{hence}\quad  \frac{H^2}{\m(H)} V^2 \sim  \left(-\tfrac{3\gamma V^5}{\mu}\right)^{1/3} \frac{1}{y\log^{1/3}\tfrac1y},
\end{equation}
as $y\to 0^+$. Therefore, fronts of \eqref{TW-slip} do not exist at all if $\lambda =0$.

\smallskip

On the other hand, if $\lambda>0$ (positive slippage), computations analogous to the ones above show that the picture is very much the same as in {\bf (Q)} and {\bf (L)}: for any $U\in \R$ ($U>0$ if $H_y(0)=0$ and $n>\frac32$) there exists a two-parameter family of quadratic fronts  satisfying \eqref{as-i-1-1} (though with a different remainder); in addition, for any $U>0$ there exists a one-parameter family of linear-log fronts satisfying \eqref{as-i-2-2}-\eqref{as-i-2-1} \citep{BKO,BSB,CG0,GGO}. A prototype case is given in Fig. \ref{generic-n2-slip}: the only notable qualitative difference is that $H_{\rm{sep}}$ are compactly supported; in fact, the unique separatrix with zero microscopic contact-angle coincides with the unique linear-log front in complete wetting, thus it is the counterpart of the maximal film $H_M$ in {\bf (M)}.

\medskip

\begin{minipage}[t]{1\textwidth}
\captionsetup{width=1\linewidth}
\centering\raisebox{\dimexpr \topskip-\height}{
  \includegraphics[width=0.6\textwidth]{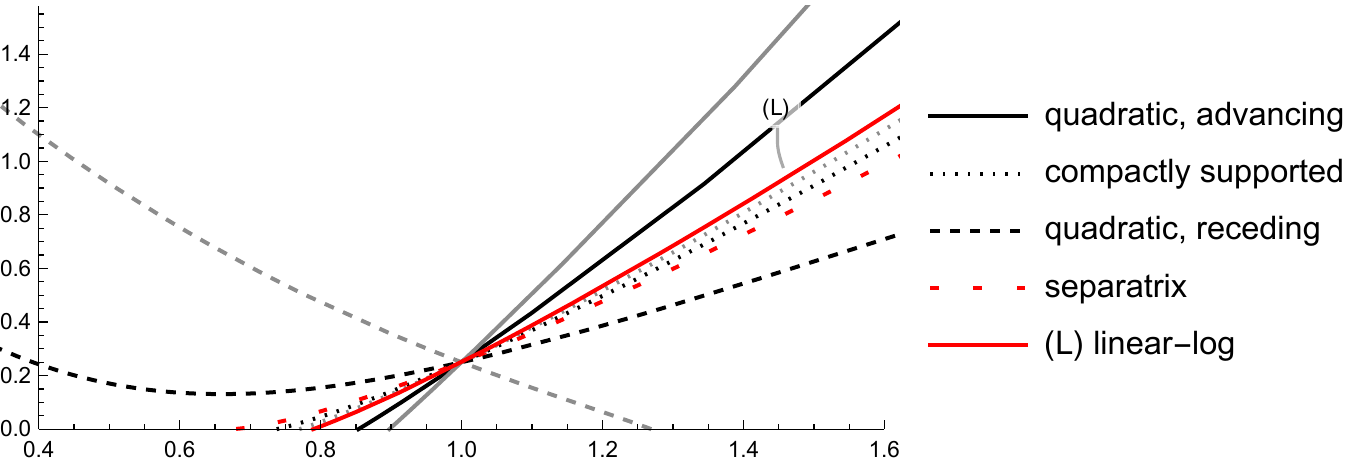}
\quad  \includegraphics[width=0.38\textwidth]{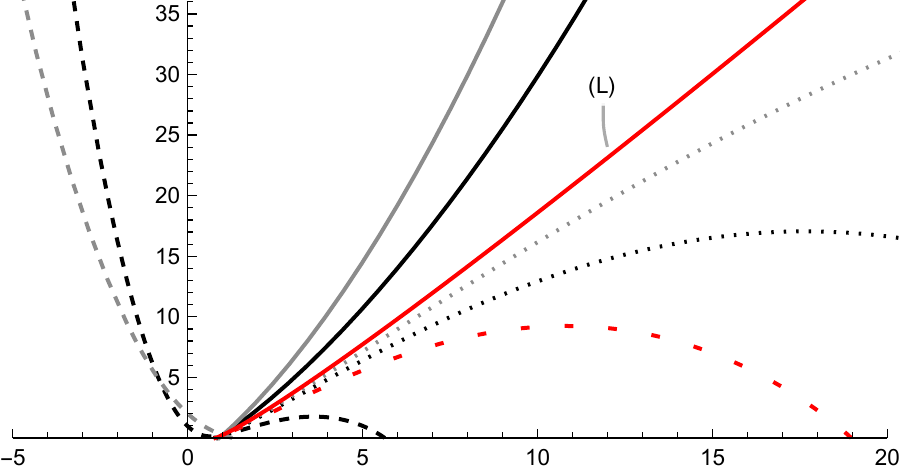}
  }
  \captionof{figure}{{\footnotesize Fronts for the thin-film equation with slippage \eqref{TW-slip} with $n=2$, $\lambda=1$, and $\alpha=H(1)=1/4$, at two different scales.}}
  \label{generic-n2-slip}
\end{minipage}

\section{Thermodynamically consistent contact-line conditions}\label{s:clc}

For exact, compactly supported solutions to the full evolution equation \eqref{TFE}, the maximal film can obviously not be taken as a selection criterion for the fronts.
In this section we will therefore identify a different criterion, which replaces contact-angle conditions in slippage models: it consists in a class of thermodynamically consistent contact-line conditions modelling friction {\em at} the contact line. The class will be identified by requiring dissipativity of the energy along the flow, in the spirit of the proposal by \cite{RE1} (see also \citet{RHE}, \citet{RE2}).

\smallskip

We assume for simplicity that $\{h>0\}$ is and remains connected for all times (i.e. we exclude coalescence or splitting of droplets):
$$
\{h>0\} = \{(t,x)\in \R_+\times\R :\ t>0, \ x\in (s_-(t),s_+(t)) \},
$$
$s_\pm(t)$ denoting the contact lines. Since $s_\pm(t)$ are unknown and \eqref{TFE} is of fourth order, three conditions are needed for well-posedness. Two of them are obvious:
\begin{equation}\label{TFE-bc}
h|_{x=s_\pm(t)}=0 \quad\mbox{and}\quad \dot s_\pm(t) = V|_{x=s_\pm(t)}.
\end{equation}
The first one defines the contact lines $s_\pm (t)$, while the second one is a kinematic condition guaranteeing no mass flux through $s_\pm(t)$. The third condition, the so-called {\it contact-line condition}, is yet debated (to a certain extent inevitably, due to the variety of material properties and configurations, which may involve surface roughness and hysteretic effects; see e.g. \citet{FK}, \citet{ADS1,ADS2}, as well as the above-mentioned reviews). The most common one amounts to prescribing a constant {\it microscopic contact angle} equal to the static one, as defined by \eqref{def-thetaS}:
\begin{equation}\label{bc-ca}
\left|h_x|_{x=s_\pm(t)}\right|=\tan\theta_S.
\end{equation}
Of particular interest to us is a relatively recent proposal by \cite{RE1} (see also \citet{RHE}, \citet{RE2}), based on consistency with the second law of thermodynamics, which also gives a robust motivation to older models by \citet{Greenspan} and \citet{Ehrhard}. In lubrication approximation \citep{CG0,CG}, when $P\equiv 0$ and $S=-\frac12\tan^2\theta_S$ (the moist case, cf. \eqref{def-thetaS}), the simplest form of the Ren-E model reads as follows$^3$\footnotetext[3]{More precisely, one should write $(h_x^2-2h h_{xx})$ in place of $h_x^2$ in the left-hand side of \eqref{bc-RE}, but it is expected that the second summand always vanishes at $x=s_\pm(t)$.}:
\begin{equation}
\label{bc-RE}
\!\!\!\!\!\pm \gamma\left(h_x^2-(\tan\theta_S)^2\right)|_{x=s_\pm(t)} = \mmu \left\{\begin{array}{lll} \dot s & \ \mbox{ if $\theta_S>0$}  & \quad \mbox{({\it partial wetting})} \\ \max\{0,\dot s\} & \ \mbox{ if $\theta_S=0$}  & \quad \mbox{({\it complete wetting})}\end{array}\right.
\end{equation}
with $\mmu>0$ a coefficient measuring friction {\em at} the contact line. Indeed, under \eqref{bc-RE}, the energy balance reads as
\begin{equation}\label{GF-stan}
\gamma \frac{\d}{\d t} \int_{s_-(t)}^{s_+(t)} \left(\tfrac12 h_x^2-S\right) \d x = -\underbrace{ \mmu (|\dot s_+|^2+|\dot s_-|^2)}_{\textrm{rate of contact-line dissipation}} - \underbrace{ \mu \int_{s_-(t)}^{s_+(t)} \frac{h^2}{\m(h)} V^2\d x}_{\textrm{rate of bulk dissipation}}
\end{equation}
(see \citet{CGw}), which shows that \eqref{bc-RE} is consistent with the second law and accounts for frictional forces {\em at} the contact line, with $\mmu \ge 0$ a friction coefficient. When $\mmu=0$ (null contact-line friction), \eqref{bc-RE} coincides with \eqref{bc-ca}.

\smallskip

We will now revisit the argument for \eqref{bc-RE} in the case of a singular potential $P$.
We base our computations on the expectation that the behavior of generic solutions coincides with that of the fronts near the contact line: letting
$$
y:=|s_\pm-x| \to 0 \quad \mbox{as $x\to s_\pm^\mp(t)$},
$$
we assume that
\begin{subequations}\label{expect}
\begin{eqnarray}
\label{expect-h}
h(t,x) &\stackrel{(\ref{Hto0})}\approx & y^\frac{2}{m+1} \quad\mbox{as $y\to 0^+$},
\\
\label{expect-hx}
h_x(t,x) &\stackrel{(\ref{as0}),(\ref{Hto0})}\approx & y^\frac{1-m}{m+1}  \quad\mbox{as $y\to 0^+$},
\\
\label{expect-V}
 V(t,x)&=& \dot s_\pm(t)(1+o(1)), \quad\mbox{as $y\to 0^+$,}
\\ \label{expect-c1}
\tfrac{1}{2}h_x^2-Q(h) &\stackrel{\eqref{as0}}=& O(1) \quad \mbox{as $y\to 0^+$,}
\\ \label{expect-pressure}
h_{xx}-Q'(h) & \stackrel{\eqref{as1b}}\approx  & o(h^{-1}) \quad \mbox{as $y\to 0^+$.}
\end{eqnarray}
\end{subequations}

Let $\eps>0$. Locally around $x=s_\pm (t)$, we may define $s_\pm^\eps(t)$ by
\begin{equation}\label{def-s}
h(t,s_\pm^\eps(t)):=\eps, \qquad\mbox{hence} \quad \left(h_t+ \dot s_\pm^\eps h_x\right)|_{x=s_\pm^\eps} =0.
\end{equation}
Using the convention $\pm a_\pm|_{x=s^\eps_\pm}:=a_+|_{x=s^\eps_+}-a_-|_{x=s^\eps_-}$ for the boundary terms, we compute:
\begin{eqnarray}
\nonumber\lefteqn{
\frac{\d}{\d t} \int_{s_-^\eps}^{s_+^\eps} \left(\tfrac12 h_x^2+Q(h)\right) \d x = \pm \dot s_\pm^{\eps} \left(\tfrac12 h_x^2+Q(h)\right)|_{x=s_\pm^\eps} +\int_{s_-^\eps}^{s_+^\eps} (h_x h_{xt}+Q'(h)h_t) \d x
}
\\
&=& \pm \dot s_\pm^\eps\left(\tfrac12 h_x^2+Q(h)\right)|_{x=s_\pm^\eps}  \pm (h_x h_t)|_{x=s_\pm^\eps}- \int_{s_-^\eps}^{s_+^\eps} (h_{xx}-Q'(h))h_t \d x \nonumber
\\ &\stackrel{(\ref{TFE})_1,(\ref{TFE-bc})}= & \pm\left[\dot s_\pm^\eps \left(\tfrac12 h_x^2+Q(h)\right) + h_x h_t \right]|_{x=s_\pm^\eps} \pm \left[(h_{xx}-Q'(h))h V\right]|_{x=s_\pm^\eps} \nonumber
\\ && - \int_{s_-^\eps}^{s_+^\eps} (h_{xx}-Q'(h))_x h V  \d x =:\pm B_{1,\eps}^\pm \pm B_{2,\eps}^\pm - \int_{s_-^\eps}^{s_+^\eps} i_{3} \d x.
\label{cv1}
\end{eqnarray}
We now notice three facts. Firstly, and crucially, the first boundary term remains bounded as $\eps\to 0$: indeed,
\begin{eqnarray}
B_{1,\eps}^\pm &\stackrel{\eqref{def-s}}= &
\dot s_{\pm}^{\eps}(t)\left(Q(h)-\tfrac12 h_x^2\right)|_{x=s^\eps_\pm(t)} \stackrel{\eqref{expect-c1}}=O(1) \quad\mbox{as $\eps\to 0$.} \label{cv2}
\end{eqnarray}
In addition, the second boundary term vanishes as $\eps\to 0$:
\begin{equation}
\label{cv3}
B_{2,\eps}^\pm\stackrel{\eqref{expect-h},\eqref{expect-V},\eqref{expect-pressure}}= o(1)
\quad\mbox{as $\eps\to 0$.}
\end{equation}
Finally, the integral on the right-hand side of \eqref{cv1} is finite: indeed,
$$
i_{3}= \tfrac{\gamma}{3\mu} h^3((h_{xx}-Q'(h))_x)^2 = \tfrac{3\mu}{\gamma} V^2 h^{-1} \stackrel{\eqref{expect-h},\eqref{expect-V}} \approx |s_\pm(t)-x|^{-\frac{2}{m+1}} \quad\mbox{as $x\to (s_\pm(t))^\mp$},
$$
which is integrable at $x=s_\pm(t)$ since $\frac{2}{m+1}<1$. Passing to the limit as $\eps\to 0$, we obtain from \eqref{def-s}-\eqref{cv3} that
\begin{equation}\label{gfpax}
\frac{\d}{\d t} E[h(t)] = \pm \gamma\dot s_\pm(t)\left(Q(h)-\tfrac12 h_x^2\right)|_{x=s_\pm(t)} - \mu\int_{s_-(t)}^{s_+(t)} \frac{h^2}{\m(h)} V^2 \d x.
\end{equation}

\begin{rem}
{\rm This formal computation is fully consistent if $m<3$. If $m\ge 3$, instead, the limit $\eps\to 0$ on the left-hand side of \eqref{gfpax} does not make sense since $E[h]\equiv +\infty$ in that case. However, \eqref{cv1} does make sense for any positive $\eps$, and the limit as $\eps\to 0$ on its right-hand side makes sense for any $m>1$.
}
\end{rem}

In order to be thermodinamically consistent, the contact-line condition has to be such that the free energy is dissipated along the flow, i.e., that the r.h.s. of \eqref{gfpax} is non-positive: this leads to the following class of contact-line conditions:
\begin{equation}\label{bc-new-gen}
\gamma \left(\tfrac12 h_x^2- Q(h)\right)|_{x=s_\pm (t)} =\pm f(\dot s_\pm)\quad\mbox{with $f$ such that $f(\dot s)\dot s\ge 0$ for all $\dot s\in \R$.}
\end{equation}
The term $f(\dot s)$ is a contribution to the dissipation which is concentrated at the contact line: in the description of \citet{Bonn}, it corresponds to the term $W_m(U)$ in formula (75). Recalling that $Q(h)=P(h)-S$, we see that the spreading coefficient $S$ enters the contact-line conditions \eqref{bc-new-gen}-\eqref{bc-RE-s-bis} in an essential way, in analogy with the contact-line conditions \eqref{bc-ca}-\eqref{bc-RE} for the slippage model.

\smallskip

The simplest choice of $f$ is a linear relation $f(\dot s)=\mmu \dot s$ ($\mmu\geq 0$), in analogy with \eqref{bc-RE}; it leads to
\begin{equation}
\label{bc-RE-s-bis}
\gamma \left(\tfrac12 h_x^2-Q(h)\right)|_{x=s_\pm(t)} = \pm \mmu \dot s_\pm(t).
\end{equation}
Substituting \eqref{bc-RE-s-bis} into \eqref{gfpax} we obtain an energy balance analogous to the one in \eqref{GF-stan}:
\begin{equation}\label{GF-new}
\frac{\d}{\d t} \int_{s_-(t)}^{s_+(t)} \gamma \left(\tfrac12 h_x^2+Q(h)\right) \d x = -\underbrace{ \mmu (|\dot s_+|^2+|\dot s_-|^2)}_{\textrm{rate\ of\ contact-line\ dissipation}} - \underbrace{ \mu\int_{s_-(t)}^{s_+(t)} \frac{h^2}{\m(h)} V^2 \d x}_{\textrm{rate\ of\ bulk\ dissipation}},
\end{equation}
which encodes a quadratic dissipation of kinetic energy through frictional forces acting {\em at} the contact line. Note that \eqref{GF-new} coincides with \eqref{fgf} if $\mmu=0$.

\smallskip

Let us comment on the contact-line condition \eqref{bc-RE-s-bis}. Its left-hand side is zero for the global minimizers $h_{min}$ discussed in \S \ref{ss:statics} (see \citet[Theorem 4.5]{DurG}). Also, its left-hand side is well defined on the fronts. Indeed, though both summands are unbounded as $x\to s_{\pm}(t)$,  their difference is not:
\begin{equation}
\label{as1-}
\left(\tfrac{1}{2}H_y^2-Q(H) \right)\stackrel{\eqref{def-Q},\eqref{hp:p0},\eqref{as0}}\sim \tfrac{A}{m-1} c_1 + S   \quad\mbox{as $H\to 0$}.
\end{equation}
For traveling waves with constant speed $V= \frac{\gamma}{3\mu}U$, $s_-(t)=-Vt$ and the contact-line condition \eqref{bc-RE-s-bis} reads as
\begin{equation}
\label{as1}
\Theta[H]:=\lim_{y\to 0^+}\left(\tfrac{1}{2}H_y^2(y)-P(H(y))\right) = \tfrac{\mmu}{3\mu} U-S.
\end{equation}
In view of \eqref{as1-}-\eqref{as1}, we expect that:

\medskip

\begin{boxlabel}
\item[{\bf (S)}] {\bf Selection criterion.} {\em Assume \eqref{def-Q}, \eqref{h:m}, \eqref{hp:p0}, and \eqref{P-large}. Let $\mmu\ge 0$.

\smallskip

\begin{boxlabeltwo}

\item[{\bf (1)}] For any $U\in \R\setminus \{0\}$, \eqref{TFE} has a {\em one-parameter} family of quadratic fronts $H$ (see {\bf (Q)}) such that \eqref{as1} holds;

\smallskip

\item[{\bf (2)}] For any $U>0$, \eqref{TFE} has a {\em unique} linear-log front $H_L$ (see {\bf (L)}) such that \eqref{as1} holds.

\end{boxlabeltwo}
}
\end{boxlabel}

\smallskip

To support the choices of both \eqref{bc-RE-s-bis} as free boundary condition and $\Theta[H]$ as selection parameter, we report numerical values of $\Theta[H_L]$ as computed for the rescaled equation \eqref{TW-m2-2}: under \eqref{scaling-m2}, the rescaled version of \eqref{as1} is, after removing hats,
\begin{equation}
\label{bc-RE-m2}
\Theta[H]=\lim_{y\to 0^+}\left(\tfrac12 H_{y}^2(y)- \tfrac{1}{m-1} H^{1-m}(y)\right)=|U|^{-2/3}\left(\tfrac{\mmu}{3\mu}U-S\right).
\end{equation}
It is apparent  from Fig. \ref{fig:Theta}(A) that $\Theta[H_L]$ monotonically covers the whole real line as linear-log solutions are spanned: in particular, a unique linear-log front can be selected such that the contact-line condition \eqref{bc-RE-m2} holds. This confirms the expectation in {\bf (S2)}. For completeness, in Fig. \ref{fig:Theta}(B) we also report numerical values of the separatrix $H_{\rm{sep}}$ which discriminates between receding and compactly supported solutions. It is apparent that $\Theta[H_{\rm{sep}}]$ increases with $\alpha$ and diverges to $-\infty$ as $\alpha\to 0^+$.

\medskip

\noindent \begin{minipage}[t]{1\textwidth}
\captionsetup{width=1\linewidth}
\centering\raisebox{\dimexpr \topskip-\height}{
  \includegraphics[width=0.4\textwidth]{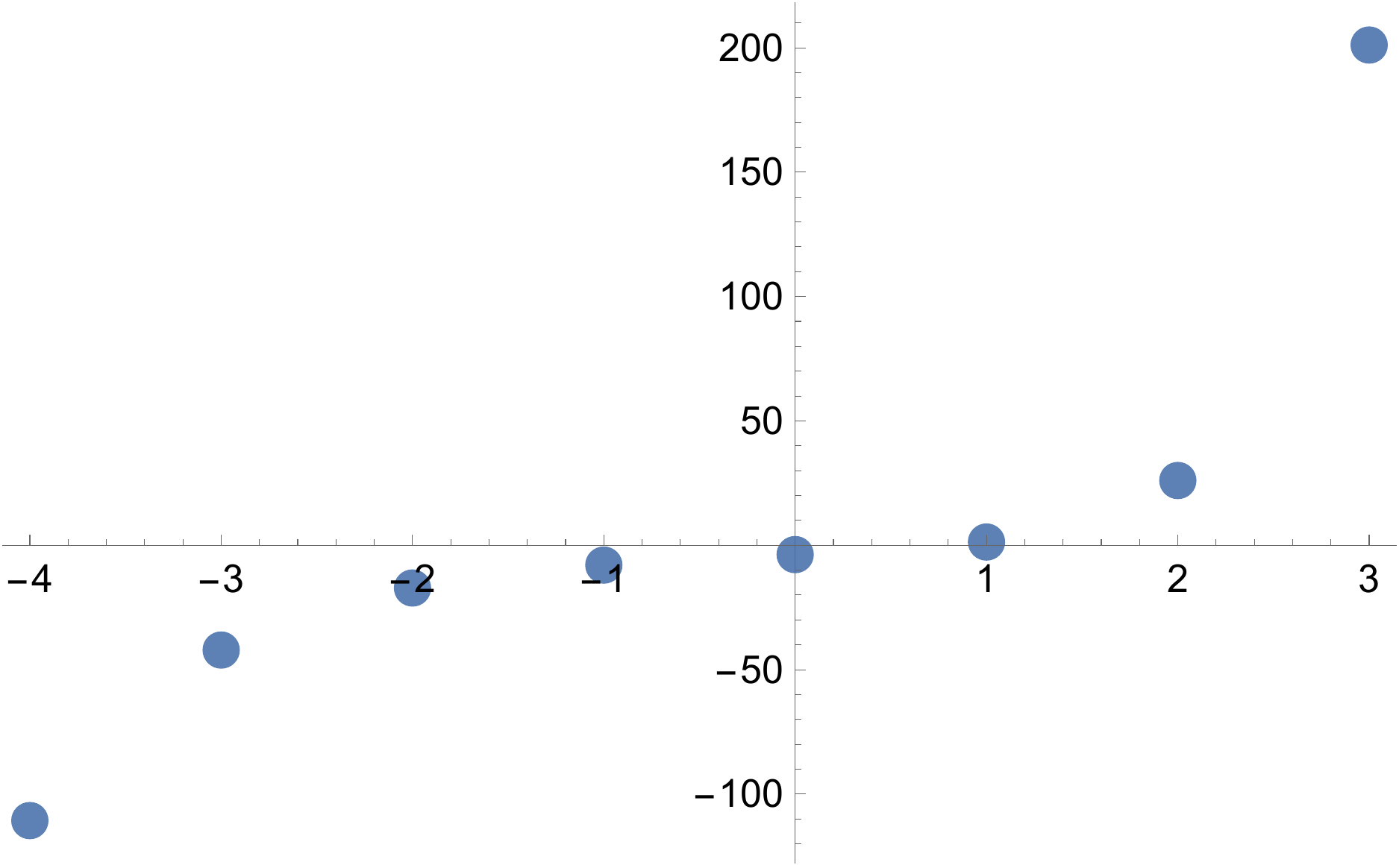}
\qquad\qquad  \includegraphics[width=0.4\textwidth]{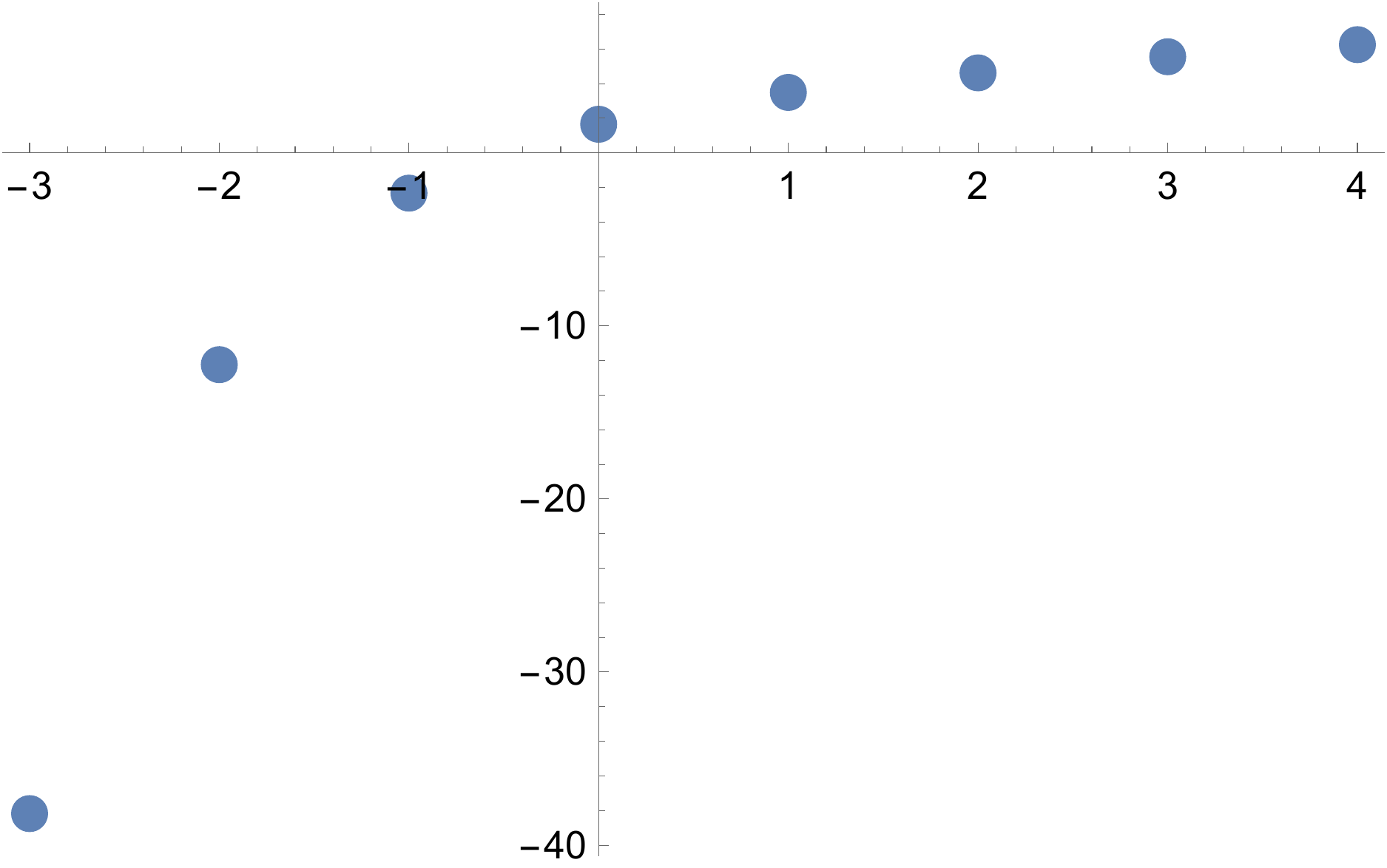}
  }
  \captionof{figure}{{\footnotesize For \eqref{TW-m2-2} with $m=2$: {\bf (A)} On the left, the values of $\Theta[H_L]$ versus $k=\log ({H_L}|_{(H_L)_{yy}=0})$. {\bf (B)} On the right, the values of $\Theta[H_{\rm{sep}}]$ versus $k=\log ({H_{\rm{sep}}}|_{(H_{\rm{sep}})_{yy}=0})$.
}}
  \label{fig:Theta}
\end{minipage}

\medskip

Combining Fig. \ref{fig:Theta}(A) with Fig. \ref{tanner-n3-m2}, it is also apparent that, as $\Theta[H_L]$ decreases, a more prominent precursor region forms ahead of the macroscopic contact line (Fig. \ref{fig:ar}). Thus, if \eqref{bc-RE-m2} is assumed as a contact-line condition, we expect that $H_L$ matches the following intuitive properties:

\smallskip

\begin{itemize}
\item for fixed $\mmu$ and $S$, a greater speed $U$ yields steeper profiles of $H_L$;

\smallskip

\item for fixed $U$ and $S$, a greater contact-line friction $\mmu$ yields steeper profiles of $H_L$;

\smallskip

\item for fixed $\mmu$ and $U$, a larger positive spreading coefficient $S$ yields gentler profiles of $H_L$.

\end{itemize}

In particular, under the contact-line condition \eqref{bc-RE-m2}, numerics suggest that larger positive spreading coefficients $S$ yield more prominent precursor regions ahead of the macroscopic contact line (Fig. \ref{fig:ar}). This agrees with the discussion in \citet{DG} for $m=3$, which is instead based on perturbations of the maximal film.

\medskip

\noindent \begin{minipage}[t]{1\textwidth}
\captionsetup{width=1\linewidth}
\centering\raisebox{\dimexpr \topskip-\height}{
  \includegraphics[width=0.4\textwidth]{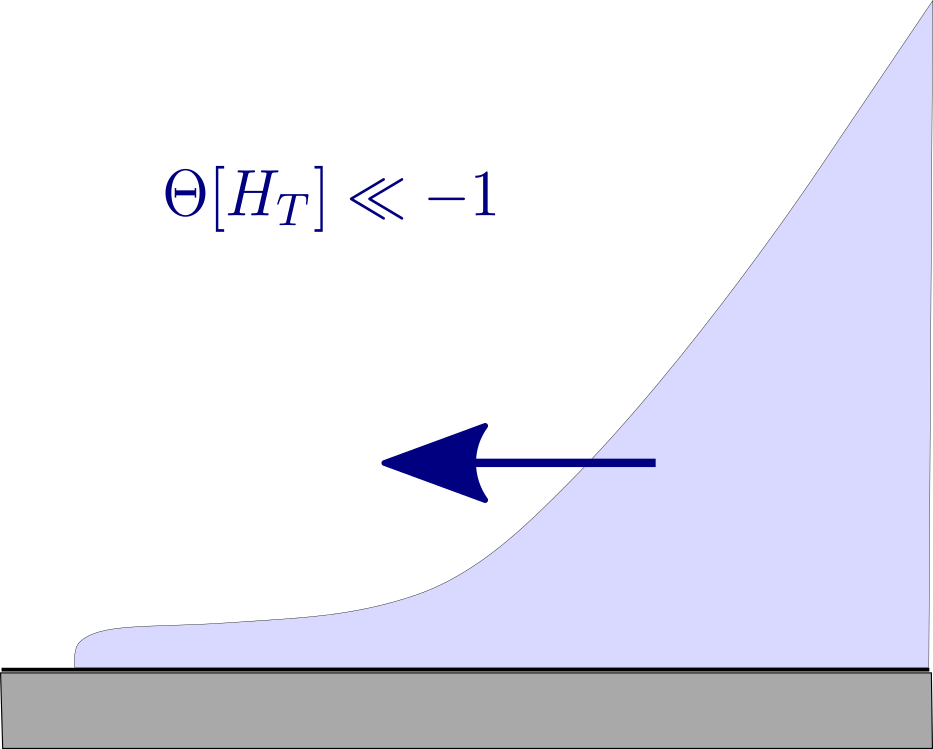}
  \qquad
  \includegraphics[width=0.4\textwidth]{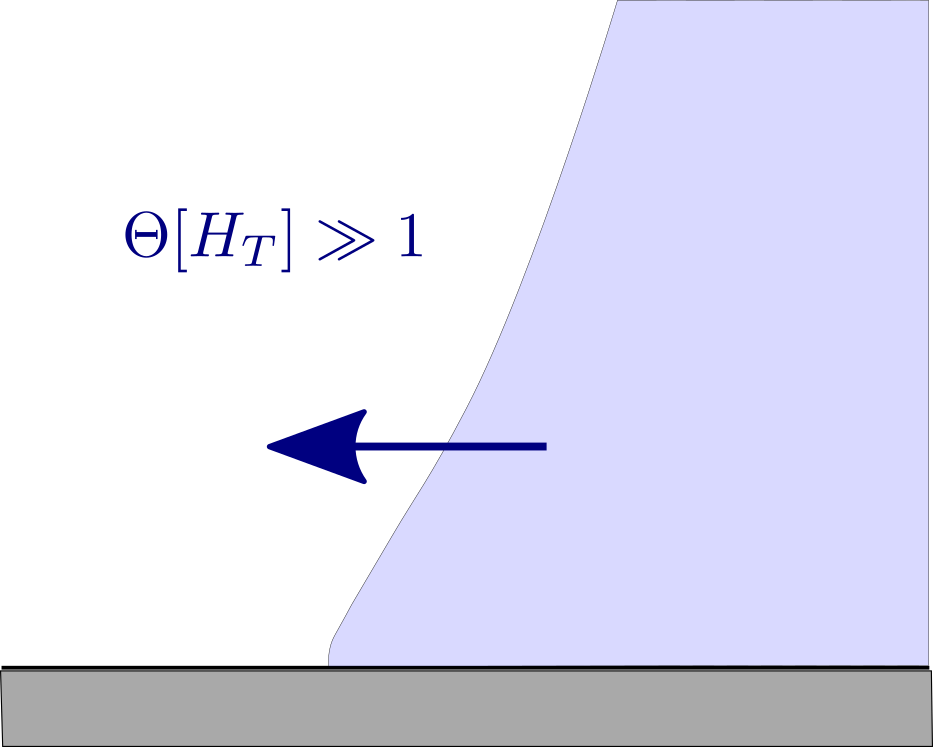}\qquad
  }
  \captionof{figure}{{\footnotesize Typical linear-log fronts $H_L$ depending on $\Theta[H_L]$. Under the contact-line condition \eqref{as1}, for given speed $U>0$ and contact-line frictional coefficient $\mmu>0$, these are typical fronts for large and positive (left), resp. large and negative (right), values of the spreading coefficient $S$ (``dry'' complete wetting, resp. partial wetting).}}
  \label{fig:ar}
\end{minipage}

\medskip

When $H_{y}|_{H_{yy}=0}>(H_L)_{y}|_{(H_L)_{yy}=0}$, resp. $H_{y}|_{H_{yy}=0}<(H_{\rm{sep}})_{y}|_{(H_{\rm{sep}})_{yy}=0}$, fronts are advancing, resp. receding, and have a quadratic profile for large $y$. In Fig. \ref{fig:Theta-gen} we report numerical values of $\Theta[H]$ for such solutions. There, it is apparent that, for each value of $\alpha=H|_{H_{yy}=0}$, $\Theta[H]$ diverges to $+\infty$ as $|H_{y}||_{H_{yy}=0}$ does. Since $\Theta[H_L]$, resp. $\Theta[H_{\rm{sep}}]$, diverge to $-\infty$ as $H|_{H_{yy}=0}\to 0$, (Fig. \ref{fig:Theta}), $\Theta[H]$ covers the whole real line as advancing, resp. receding, traveling waves are spanned, thus confirming {\bf (S1)}. There is, however, a difference between advancing and receding fronts, since receding ones may be non-monotonic (Figg. \ref{generic-n3-m2} and \ref{generic-n3-m3}); note that the same happens for slippage models (Fig. \ref{generic-n2-slip}). This reflects into a lack of monotonicity of $\Theta[H]$ for fixed $\alpha$ on receding waves. Note that the minimum of $\Theta[H]$ appears to be on the left half-plane, hence the corresponding waves have negative derivative at their inflection point, thus they are monotonic. Therefore the branch of receding fronts emanating from $\Theta[H]=+\infty$ consists of monotone ones. This phenomenon might be related to the existence of a ``limiting speed'' of receding fronts calculated by \citet{Eggers-forced,Eggers20051} in the slippage case.

\medskip

\noindent \begin{minipage}[t]{1\textwidth}
\captionsetup{width=1\linewidth}
\centering\raisebox{\dimexpr \topskip-\height}{
  \includegraphics[width=0.45\textwidth]{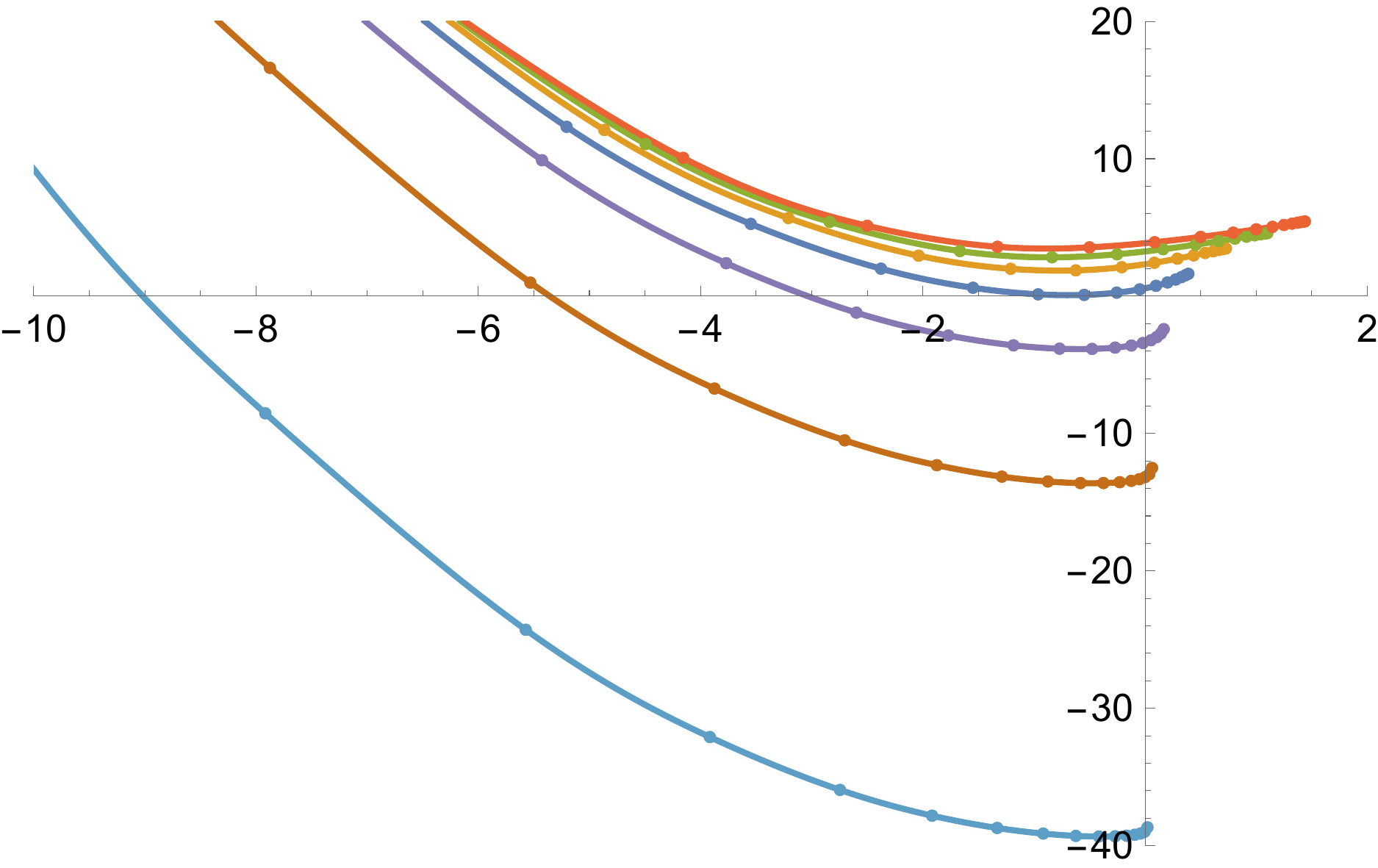}
\qquad  \includegraphics[width=0.45\textwidth]{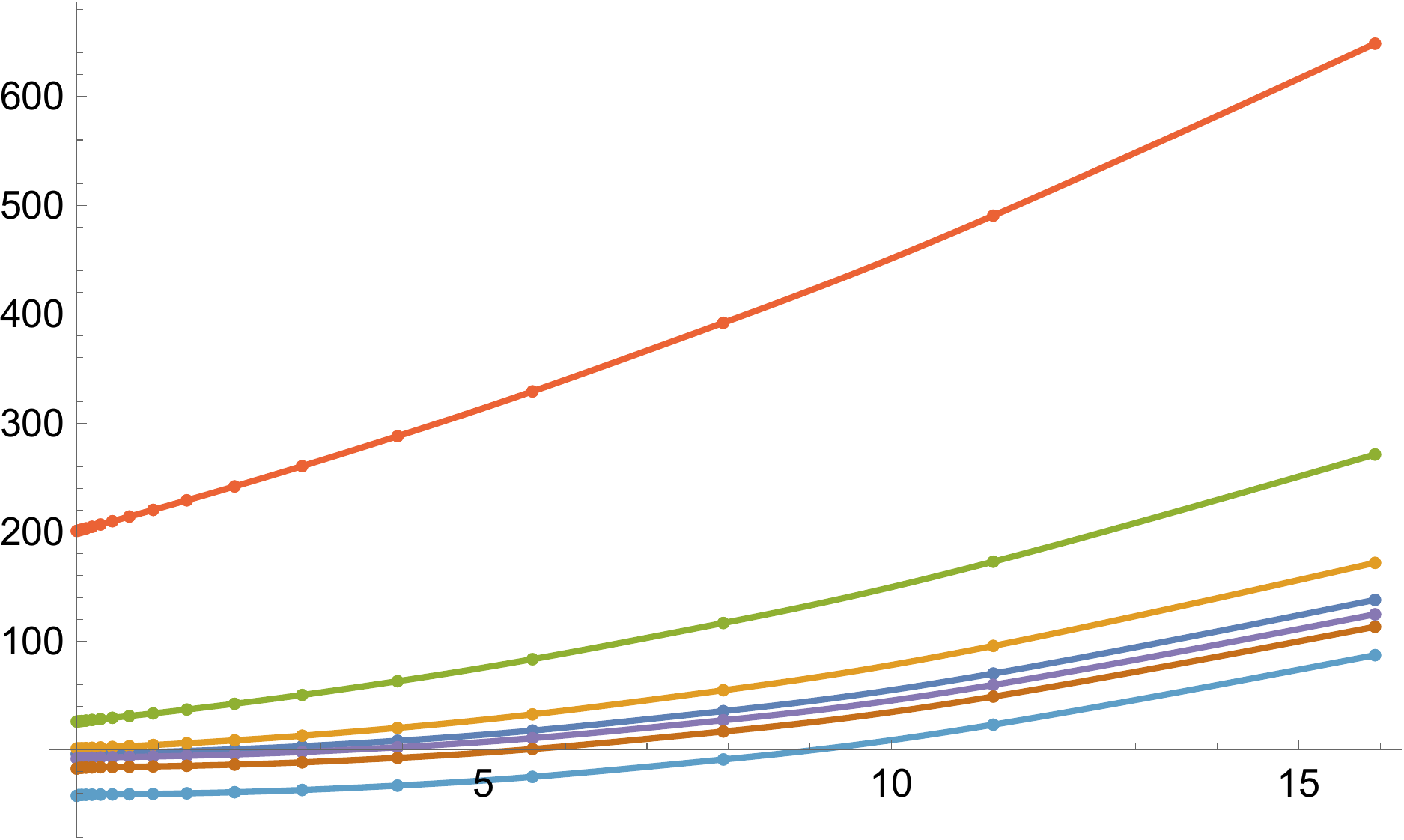}
  }
  \captionof{figure}{{\footnotesize For $m=2$: On the left, the value of $\Theta[H]$ versus ${H_{y}}|_{H_{yy}=0}$ for receding fronts. On the right, the value of $\Theta[H]$ versus ${H_{y}}|_{H_{yy}=0}-(H_L)_{y}|_{(H_L)_{yy}=0}$ for advancing fronts. In both cases, $k=\log H|_{H_{yy}=0}$ ranges from $-3$ (bottom) to $3$ (top).
}}
  \label{fig:Theta-gen}
\end{minipage}

\section{Conclusions and open questions}\label{s:cop}

We have discussed thin-film models under singular potentials:
\begin{equation}\label{concl1}
h_t + \tfrac{\gamma}{3\mu}(h^3(h_{xx}-P'(h))_x)_x=0
\end{equation}
with
\begin{equation}\label{concl2}
P(h) \sim \tfrac{A}{m-1} h^{1-m} \quad\mbox{as $h\to 0^+$ } \ \mbox{ with \ $m>1$,}\quad A>0, \quad P(0)=P(+\infty)=0.
\end{equation}
Based on formal arguments supported by numerical evidence, we have argued that singular potentials generically solve the contact-line paradox, in the sense that \eqref{concl1} admits for any value of $m>1$:

\medskip

\begin{boxlabel}
\item[{\bf (Q)}] a two-parameter family of both advancing and receding traveling-waves with finite rate of dissipation at the contact line;
\item[{\bf (L)}] a one-parameter family of advancing ``linear-log'' travelling waves displaying a logarithmically corrected linear behavior in the bulk, also with finite rate of dissipation at the contact line.
\end{boxlabel}

\medskip

In agreement with mass-constrained steady states, travelling waves have finite energy if and only if $m<3$, whereas for $m\ge 3$ a cut-off at a molecular-size length-scale is necessary (cf. e.g. \eqref{P-model}). However, the qualitative properties of the waves are the same for any $m>1$. Our formal arguments also suggest that for $m\ge 2$ and any positive speed there exists:

\medskip

\begin{boxlabel}
\item[{\bf (M)}] a unique maximal film, i.e., an advancing linear-log travelling wave $H(y)$ which decays to zero as $y\to -\infty$ instead of touching down to zero at a finite point.
\end{boxlabel}

\medskip

Intermolecular potentials thus stand as a possible solution to the contact-line paradox, alternative to the most common one, given by slippage models:
\begin{equation}\label{concl4}
h_t + \tfrac{\gamma}{3\mu}((h^3+\lambda^{3-n} h^n)h_{xxx})_x=0, \qquad \lambda>0, \ n\in (0,3).
\end{equation}
For equation \eqref{concl4}, the classification of travelling waves is analogous to {\bf (Q)} and {\bf (L)} (cf. \S \ref{ss:slip}); however, the microscopic contact angle $h_x|_{\{h=0\}}$ may be used as a parameter to span them, thus selecting a unique advancing linear-log front. This is impossible for \eqref{concl1}-\eqref{concl2}, since in that case the microscopic contact angle is always $\pi/2$. Here, we have also proposed a class of thermodynamically consistent contact-line conditions, which replaces contact-angle ones and is expected to single out a unique advancing linear-log front with a contact line. The simplest among such conditions reads as
\begin{equation}\label{concl3}
\Theta[h(t)]:= \left(\tfrac12 h_x^2-P(h)\right)|_{x=s_\pm(t)} = \pm \tfrac{\mmu}{\gamma} \dot s_\pm(t)-S,
\end{equation}
where $\{h(t)>0\}=(s_-(t),s_+(t))$, $\mmu \ge 0$ is a contact-line frictional coefficient, and $S$ is the non-dimensional spreading coefficient. We expect that:

\medskip

\begin{boxlabel}
\item[{\bf (S)}] for any value of the speed, \eqref{concl3}  selects a one-parameter family of fronts in {\bf (Q)} and a unique linear-log front in {\bf (L)}.
\end{boxlabel}

\medskip

\noindent Numerical evidence also suggests that linear-log fronts are steep for $\Theta[h]\gg 1$, whereas they display a precursor region ahead of the macroscopic contact line for $\Theta[h]\ll -1$. Therefore, we expect that \eqref{concl3} yields precursor regions for large positive values of the spreading coefficient.

\smallskip

The above four observations issue quite a few challenges.

\smallskip

\begin{itemize}

\item The rigorous validation of {\bf (Q)}, {\bf (L)} and {\bf (M)} is highly desirable. Once the asymptotics in {\bf (Q)} and {\bf (L)} have been proved, we expect that {\bf (S)} will follow as a byproduct.

\smallskip

\item For relatively large values of $S$, it would be interesting to quantify height and length of the precursor region, which appears to exist ahead of the ``macroscopic'' contact line, in relation to the contact-line condition \eqref{concl3}.

\smallskip

\item It would be very useful to have numerical simulations and/or matched asymptotic studies available for generic solutions to \eqref{concl1} with the contact-line condition \eqref{concl3}, for potentials $P$ of the general form \eqref{concl2}. Of particular interest would be the (in)stability of advancing/receding traveling waves, the detection of scaling laws --such as the Voinov-Cox-Hocking logarithmic correction to Tanner's law--, an estimate of the rate of convergence to equilibria, and an insight on the evolution of the precursor region for relatively large values of $S$.

\smallskip

\item Based on {\bf (S)}, for potentials $P$ of the form \eqref{concl2}, we conjecture that for any non-negative $h_0\in H^1(\R)$ such that $E[h_0]<+\infty$ there exists a unique solution to \eqref{concl1} with the contact-line conditions \eqref{TFE-bc} and \eqref{concl3}. A difficult but extremely interesting task would be to develop a well-posedness theory at least when $h_0$ is a perturbation of a traveling wave, in the spirit of \citep{GKO,GGKO,K1,K2,KM1,KM2,Gn1,GP}.

\smallskip

\item \citet{DSS} recently analyzed the $\Gamma$-convergence of the energy $E$ in the limit $m\to 3^-$ and vanishing $A$ ($A=A_m\approx (3-m)^2\to 0$ as $m\to 3$). It would be interesting to understand how this scaling limit extends to the dynamic framework.

\smallskip

\item In relation to the issue of infinite energy for $m\ge 3$, it would be interesting to explore the effect of taking the full curvature operator into account, i.e., replacing $1+\frac12 h_x^2$ with $\sqrt{1+h_x^2}$. We are aware of only a few studies \citep{NC1,NC2,NC3,NC4}, dealing with the statics under convex potentials.

\smallskip

\item  Interesting, though of seemingly lesser impact, would also be to face the above challenges for more general mobilities, e.g. of the form  $\m(h)=\frac13(h^3+\lambda^{3-n} h^n)$ (cf. alos Remark \ref{rem:n}).

\end{itemize}

\section{Appendix}

The limitation \eqref{dual} of the energy in \eqref{def-E} is well known, though with slightly different formulations; see e.g. \citet[Lemma 2.10]{DurG} in $\R$ and \citet[Theorem 2]{LM} on bounded domains. For completeness, here we provide a comprehensive statement in $\R$. Note that $E$ is a singular version of the \citet{AltP} functional.

\begin{lemma}\label{lem:repeat}
Let $E$ as in \eqref{def-E} and let $Q:[0,+\infty)\to \R$ be such that $Q\in C((0,+\infty))$, $Q(0)=0$ and $Q(h)\gtrsim h^{-2}$ as $h\to 0^+$. Let $h\in H^1_{\text{loc}}(\R)$ such that $h\ge 0$ and $\int_{\R} h\in (0,+\infty]$. If $E[h]$ is finite, then $h>0$ in $\R$ and $h\not\to 0$ as $|x|\to +\infty$.
\end{lemma}

\begin{proof}
Since the mass is positive, $h\not\equiv 0$.
For the first statement, assume by contradiction that $h^{-1}(\{0\})\ne \emptyset$. Since $h\not\equiv 0$, $x_0\in \R$ exists such that $h(x_0)=0$ and a left or right neighborhood ${\mathcal U}$ of $x_0$ exist such that $h>0$ in ${\mathcal U}$. Now, $h\in H^1({\mathcal U})$ implies that $h(x)\lesssim |x-x_0|^{1/2}$ in ${\mathcal U}$; therefore
$Q(h(x))\geq h(x)^{-2}\gtrsim |x-x_0|^{-1}$ in ${\mathcal U}$, whence a contradiction:
$$
\int_{\mathcal U} \left(\tfrac 12h_x^2 + Q(h(x))\right)\d x \ge \int_{\mathcal U} Q(h(x)) \d x\gtrsim \int_{\mathcal U} |x-x_0|^{-1} \d x = +\infty.
$$
The second statement is straightforward: if by contradiction $h\to 0$ as, say, $x\to +\infty$, then we would have $0<h\ll 1$ for $x\gg 1$, hence $Q(h)\gg 1$ for all $x\gg 1$, which obviously contradicts the finiteness of the energy.
\end{proof}

\tocless\subsection*{Acknowledgements} Authors have been partially supported by GNAMPA of INdAM (Project ``Analisi di fenomeni di wetting in presenza di potenziali singolari''). The first author has also been supported by PON Ricerca e Innovazione D.M. 1062/21.

\tocless\subsection*{Data availability statement} All data generated or analysed during this study are included in this published article.

\footnotesize

\bibliography{bibliomod}{}
\bibliographystyle{abbrvnat}


\end{document}